\documentclass[reqno, 10pt]{article}

\usepackage{fullpage}

\usepackage[centertags]{amsmath}
\usepackage{amsmath}
\usepackage{amsfonts}
\usepackage{amssymb}
\usepackage{amsthm}
\usepackage{enumerate}
\usepackage{subfig}
\usepackage{hyperref}

\usepackage{algorithm,algorithmic}
\usepackage{dsfont}

\usepackage{newlfont}
\usepackage{color}

\usepackage{authblk}

\usepackage{stmaryrd}
\SetSymbolFont{stmry}{bold}{U}{stmry}{m}{n}

\usepackage{bbm}
\include{amsthm_sc}

\usepackage{stmaryrd}
\usepackage{mathrsfs}

\usepackage{fancyhdr}
\usepackage{graphicx}
\usepackage{fancybox}
\usepackage{setspace}
\usepackage{cleveref}

\newtheorem{thm}{Theorem}
\newtheorem{cor}[thm]{Corollary}
\newtheorem{lem}[thm]{Lemma}
\newtheorem{prop}[thm]{Proposition}

\newtheorem{defn}[thm]{Definition}
\newtheorem{assump}[thm]{Assumption}

\newtheorem{rem}[thm]{Remark}
\newtheorem{nrem}[thm]{Notational Remark}

\newcommand{\R} {\mathbb{R}}
\newcommand{\C} {\mathbb{C}}
\newcommand{\N} {\mathbb{N}}

\newcommand{\E} {\mathbb{E}}

\newcommand{\p} {\mathbb{P}}

\DeclareMathOperator{\diag}{diag}

\DeclareMathOperator{\Tr}{Tr}

\DeclareMathOperator{\err}{\mathrm{err}}
\DeclareMathOperator{\erfc}{\mathrm{erfc}}

\newcommand{\caL}{{\mathcal L}}

\newcommand{\caN}{{\mathcal N}}
\newcommand{\caO}{{\mathcal O}}
\newcommand{\caP}{{\mathcal P}}

\newcommand{\bbC}{{\mathbb C}}

\newcommand{\bbP}{{\mathbb P}}

\newcommand{\bbR}{{\mathbb R}}

\newcommand{\bsx}{{\boldsymbol x}}

\newcommand{\bsH}{{\boldsymbol H}}

\newcommand{\wt}{\widetilde}

\newcommand{\beq}{ \begin{equation} }
\newcommand{\eeq}{ \end{equation} }

\newcommand{\dd}{\mathrm{d}}
\newcommand{\ii}{\mathrm{i}}

\renewcommand{\P}{\bbP}

\newcommand{\tM}{\widetilde{M}}

\newcommand{\fh}{F^H}
\newcommand{\gh}{G^H}

\newcommand\absv[1]{\left\vert\,#1\,\right\vert}

\makeatletter
\def\blfootnote{\xdef\@thefnmark{}\@footnotetext}
\makeatother

\numberwithin{equation}{section} 
\numberwithin{thm}{section}

\title{Weak Detection in the Spiked Wigner Model with General Rank}

\author{Ji Hyung Jung\footnote{Department of Mathematical Sciences, KAIST, Daejeon, 34141, Korea
		\newline email: \texttt{jhjung66@kaist.ac.kr}}
	, Hye Won Chung\footnote{School of Electrical Engineering, KAIST, Daejeon, 34141, Korea
		\newline email: \texttt{hwchung@kaist.ac.kr}}
	, and Ji Oon Lee\footnote{Department of Mathematical Sciences, KAIST, Daejeon, 34141, and School of Mathematics, KIAS, Seoul, 02455, Korea
		\newline email: \texttt{jioon.lee@kaist.edu}}}

\date{\today}

\begin{document}
	
	\maketitle
	
\begin{abstract}%
	We study the statistical decision process of detecting the signal from a `signal+noise' type matrix model with an additive Wigner noise. We propose a hypothesis test based on the linear spectral statistics of the data matrix, which does not depend on the distribution of the signal or the noise. The test is optimal under the Gaussian noise if the signal-to-noise ratio is small, as it minimizes the sum of
	the Type-I and Type-II errors. Under the non-Gaussian noise, the test can be improved with an entrywise transformation to the data matrix. We also introduce an algorithm that estimates the rank of the signal
	when it is not known a priori.
	
\end{abstract}


\section{Introduction}\label{sec:intro}
The spiked Wigner model is one of the most natural low-rank models of `signal-plus-noise' type. In this model, the data matrix is of the form
\beq \label{eq:spiked_Wigner}
M = \sqrt{\lambda} XX^T + H,
\eeq
where the spike $X$ is an $N \times k$ matrix whose column vectors are $L^2$-normalized, $H$ is an $N \times N$ Wigner matrix (see Definition \ref{def:Wigner}), and $\lambda$ corresponds to the signal-to-noise ratio (SNR). In this paper, we focus on the hypothesis tests for detecting presence of a signal and determining the rank of a signal (when we have some information of the rank), which are called the weak detection, from a given spiked Wigner matrix where SNR $\lambda$ is below a threshold so that a reliable detection is not feasible. We prove the optimal error of the weak detection in certain cases and propose a universal test that achieve the optimal error. We also introduce a test to estimate the rank of the spike when the prior information on the rank is not known.

\textbf{Rank-$1$ spiked Wigner matrix:} 
In the simplest case of the spiked Wigner model, the signal $\bsx$ is a vector and the spiked Wigner matrix is of the form
\beq \label{eq:rank-1}
\sqrt{\lambda} \bsx \bsx^T + H.
\eeq
The spectral properties of rank-$1$ spiked Wigner matrix have been extensively studied in random matrix theory (\cite{Peche2006,FeralPeche2007,CapitaineDonatiFeral2009,BenaychNadakuditi2011}), and the detection limits have been investigated in statistical learning theory (\cite{onatski2013asymptotic,onatski2014signal,JohnstoneOnatski2015,montanari2015limitation,Barbier2016,LelargeMiolane2016,miolane2017fundamental,BanerjeeMa2017,Perry2018,AlaouiJordan2018,Chung-Lee2019}). The model is also applied to various problems such as community detection (\cite{Abbe2017}) and submatrix localization (\cite{Butucea2013}).

In the rank-$1$ spiked Wigner matrix with Gaussian noise, assuming the signal is drawn from a distribution, called the prior, the signal is not reliably detectable if the SNR $\lambda$ is below a certain threshold (\cite{montanari2015limitation}). With the normalization $\|\bsx\|_2 = 1$ and $\| H \| \to 2$ as $N \to \infty$, the threshold is $1$ for a general class of priors, including spherical, Rademacher, and any i.i.d. prior with a sub-Gaussian bound (\cite{Perry2018}). On the other hand, the signal can be estimated if the SNR is above the threshold (\cite{Barbier2016,LelargeMiolane2016,miolane2017fundamental}).

In the subcritical case where the signal is not reliably detectable, it is natural to consider a hypothesis test on the presence of the signal.
As asserted by Neyman--Pearson lemma, the likelihood ratio (LR) test is optimal in the sense that it minimizes the sum of the Type-I error and the Type-II error. It was proved in \cite{AlaouiJordan2018} that this sum converges to
\beq
\erfc \left( \frac{1}{4} \sqrt{-\log(1-\lambda)} \right)
\eeq
when the variance of $H_{ii}$ is $2$ and hence $H$ is a Gaussian Orthogonal Ensemble (GOE). Though optimal, the LR test is not efficient, and it is desirable to construct a test that does not depend on information about the distribution of the signal, called prior, which is typically not known in many practical applications. In \cite{Chung-Lee2019}, an optimal and universal test was proposed, which is based on the linear spectral statistics (LSS) of the data matrix, a linear functional defined as
\beq\label{eq:LSS}
L_N(f)=\sum_{i=1}^{N}f(\mu_i)
\eeq
for a given function $f$, where $\mu_1,\cdots\mu_N$ are the eigenvalues of the data matrix. For other results on the rank-$1$ spiked Wigner model, we refer to \cite{AlaouiJordan2018,Perry2018,Chung-Lee2019} and references therein.

\textbf{Main problem:} We consider the detection problem in the general spiked Wigner model in \eqref{eq:spiked_Wigner}. Let us denote by $\bsH_1$ and $\bsH_2$ the hypotheses
\beq \label{eq:hyp}
\bsH_1 : k=k_1, \qquad \bsH_2 : k=k_2
\eeq
for distinct non-negative integers $k_1<k_2$. While it may seem obvious, it has not been even known in the simple case $k_1=0$ whether the detection becomes easier as $k_2$ increases. The principal component analysis (PCA), which is one of the most commonly used techniques to analyze the matrix model, can detect the presence of the signal if and only if $\lambda>1$, regardless of $k_2$ (\cite{BenaychNadakuditi2011}). Our goal is to construct an efficient algorithm for a hypothesis test between $\bsH_1$ and $\bsH_2$ that is universal, optimal, and data-driven as in \cite{Chung-Lee2019}. 


\textbf{Main contributions:} 
We propose a test based on the central limit theorem (CLT) of the LSS analogous to the one introduced in \cite{Chung-Lee2019}. The test is universal, and the various quantities in it can be estimated from the observed data without any prior knowledge on the signal or the noise. Furthermore, we also show that the proposed test can be improved by adapting the entrywise transformation in \cite{Perry2018}.

To prove the optimality of the proposed test, we compare the error of the proposed test with that of the LR test. Adapting the strategies of \cite{LelargeMiolane2016,AlaouiJordan2018}, we study the LR of the spiked Wigner matrices for small $\lambda$ and show that the log-LR converges to a Gaussian whose mean and variance depend on $k_1$ and $k_2$ (see Theorems \ref{thm:main} and \ref{thm:LR_err}).

An important issue when applying the spiked Wigner matrix is that the rank of the spike must be known a priori. Viable solutions to resolve the issue in the context of the community detection were suggested in \cite{bickel2016hypothesis,lei2016goodness}, which can work for any spiked Wigner matrices whenever $\lambda \gg 1$. However, their methods, which are spectral in nature, are not applicable in the regime $\lambda < 1$ regardless of the rank of the spike. With the CLT of the LSS, we also introduce a test for rank estimation that does not require the prior information on the rank of the signal.

The main mathematical achievement of the current paper is the CLT for the LSS of spiked Wigner matrices with general ranks. 
For a rank-$1$ spiked Wigner matrix, the CLT was first proved for a special spike $\frac{1}{\sqrt{N}}(1, 1, \dots, 1)^T$ in \cite{Baik-Lee2017} and later extended for a general rank-$1$ spike by comparison with the special case (\cite{Chung-Lee2019}). However, the proof in \cite{Baik-Lee2017} is not readily extended to the spiked Wigner matrices with higher ranks. In this paper, we overcome the difficulty by introducing a direct interpolation between the spiked Wigner matrix and the corresponding Wigner matrix without a spike and tracking the change of the LSS.

\subsection{Model}

The data matrix we consider is a rank-$k$ spiked Wigner matrix, which is defined as follows:

\begin{defn}[Wigner matrix] \label{def:Wigner}
	An $N \times N$ symmetric random matrix $H = (H_{ij})$ is a (real) Wigner matrix if $H_{ij}$ ($i, j = 1, 2, \dots, N$) are independent real random variables such that
	\begin{itemize}
		\item All moments of $H_{ij}$ are finite and $\E[H_{ij}]=0$ for all $i \leq j$.
		\item For all $i<j$, $N \E[H_{ij}^2]=1$, $N^{\frac{3}{2}} \E[H_{ij}^3]=w_3$, and $N^2 \E[H_{ij}^4]=w_4$ for some $w_3, w_4\in \R$. 
		\item For all $i$, $N\E[H_{ii}^2]=w_2$ for some constant $w_2\geq 0$. 
	\end{itemize}
\end{defn}

\begin{defn}[(rank-$k$) Spiked Wigner matrix] \label{def:spiked_Wigner}
	An $N \times N$ matrix $	M = \sqrt{\lambda} XX^T + H$ is a spiked Wigner matrix with a spike $X$ and the SNR $\lambda$ if $H$ is a Wigner matrix and a $N\times k$ signal matrix $X = [\bsx(1), \bsx(2), \dots, \bsx(k)]$ with $\bsx(i) \in \R^N$ and $\|\bsx(i)\|_2=1$ for $i=1, 2, \dots, k$.
\end{defn}

For the analysis of the data matrix with Gaussian noise, 
we use a spiked Wigner matrix with the following normalization.

\begin{defn}[(rank-$k$) Spiked Gaussian Wigner matrix] \label{def:Gaussian}
	An $N\times N$ matrix $Y=\sqrt{\frac{\lambda}{N}}X^{\ast}X^{\ast T}+W$ is a spiked Gaussian Wigner matrix with the SNR $\lambda$ with the spike $X^{\ast}$ and SNR $\lambda$ if $W=\sqrt{N}H$ for a Wigner matrix $H$ and $X^{\ast}=\left[\bsx^{\ast}(1), \bsx^{\ast}(2), \cdots, \bsx^{\ast}(k)\right]\in\R^{N\times k}$. We assume that the columns of the spike matrix $x^{\ast}_i(\ell)$ are i.i.d. with a prior distribution $\caP$ having bounded support.
\end{defn}

\subsection{Other related works}

The spiked Wigner model can be generalized to $p$-tensor models ($p \geq 3$). With the rank-$1$ spherical spike, the phase transition was proved in \cite{montanari2015limitation,richard2014statistical} that there exist $\lambda_{-}\le\lambda_{+}$ such that detection is impossible for $\lambda<\lambda_{-}$ but is possible for $\lambda>\lambda_{+}$. The tensor models with multiple spikes were considered in \cite{lesieur2017statistical,barbier2017stochastic,Chen2018} where i.i.d. signals are sampled from a joint of centered priors with finite variance, and it was further generalized to the non-symmetric setting in \cite{barbier2017layered}. 

\subsection{Organization of the paper}

The rest of the paper is organized as follows:
In Section~\ref{sec:main}, we propose algorithms for LSS-based tests and a test for rank estimation, and analyze their performance. In Section~\ref{sec:Gaussian}, we prove the Gaussian convergence of the log-LR of the spiked Gaussian Wigner model, which asserts the optimality of the proposed test. In Section~\ref{sec:LSS}, we state general results on the CLT for the LSS. We conclude the paper in Section~\ref{sec:summary} with the summary of our works and future research directions. In Appendix~\ref{sec:ex}, we consider examples of spiked Wigner matrices and provide results from numerical experiments.
In Appendices \ref{app:thm} and \ref{app:CLT}, we provide technical details of the proofs.

\section{Main results} \label{sec:main}




\subsection{Hypothesis testing based on LSS for spiked Wigner Matrices} \label{sec:non-Gaussian}

Recall the hypotheses defined in \eqref{eq:hyp}. 
In \cite{Chung-Lee2019}, the following test statistic was considered for the case $k_1=0$, $k_2=1$:
\beq \begin{split} \label{eq:L_lambda}
	L_{\lambda} &= - \log \det \left( (1+\lambda) I - \sqrt{\lambda} M \right) +  \frac{\lambda N}{2} \\
	&\qquad + \sqrt{\lambda} \left( \frac{2}{w_2} - 1 \right) \Tr M 
	+ \lambda \left( \frac{1}{w_4-1} - \frac{1}{2} \right) (\Tr M^2 - N).
\end{split} 
\eeq
If there is no signal present, 
$L_{\lambda} \Rightarrow \caN(m_0, V_0)$,
where
\beq \begin{split} \label{eq:m_0}
	m_0 = -\frac{1}{2} \log(1-\lambda) + \left(\frac{w_2 -1}{w_4-1} -\frac{1}{2} \right) \lambda + \frac{(w_4 -3) \lambda^2}{4},
\end{split} \eeq
\beq \begin{split} \label{eq:V_H}
	V_0 = -2 \log(1-\lambda) + \left( \frac{4}{w_2}-2 \right) \lambda + \left( \frac{2}{w_4-1} - 1 \right) \lambda^2.
\end{split} \eeq

For a rank-$k$ spiked Wigner matrix, $L_{\lambda}$ also converges to a Gaussian with the same variance $V_0$ but an altered mean $m_k$. The following is the precise statement for the limiting distribution of $L_{\lambda}$.

\begin{thm} \label{thm:main-weak}
	Let $M$ be a rank-$k$ spiked Wigner matrix with a spike $X$ as in Definition \ref{def:spiked_Wigner} with $0 < \lambda < 1$. Denote by $\mu_1 \geq \mu_2 \geq \dots \geq \mu_N$ the eigenvalues of $M$. Then, 
	\beq
	L_{\lambda} \Rightarrow \caN(m_k, V_0)\,,
	\eeq
	where the variance $V_0$ is as in \eqref{eq:V_H} and the mean $m_k$ is given by
	\beq \begin{split} \label{eq:m_k}
		m_k = m_0 + k \left[ -\log(1-\lambda) + \left( \frac{2}{w_2} - 1 \right) \lambda + \left( \frac{1}{w_4-1} - \frac{1}{2} \right) \lambda^2 \right].
	\end{split} 
	\eeq
\end{thm}

\begin{proof}
	Theorem \ref{thm:main-weak} directly follows from Theorem \ref{thm:CLT} in Section \ref{sec:LSS}.
\end{proof}

We can construct a hypothesis test (between $\bsH_1$ and $\bsH_2$) based on Theorem \ref{thm:main}, which we describe in Algorithm \ref{alg:ht}. In this test, for a given data matrix $M$, we compute $L_{\lambda}$ and compare it with the critical value $m_{\lambda}$, defined as
\beq \begin{split} \label{eq:m_lambda}
	m_{\lambda} := \frac{m_{k_1}+m_{k_2}}{2} 
	&= -\frac{k_1 + k_2 + 1}{2} \log(1-\lambda) + \left(\frac{w_2 -1}{w_4-1} + \frac{k_1 + k_2}{w_2} - \frac{k_1 + k_2 + 1}{2} \right) \lambda \\
	&\qquad  + \left( \frac{w_4 - k_1 - k_2 - 3}{4} + \frac{k_1 + k_2}{2(w_4-1)} \right) \lambda^2.
\end{split} \eeq
In Theorem \ref{thm:optimize}, we prove that the error of the CLT-based test is minimized with the test statistic $L_{\lambda}$ also for rank-$k$ spiked Wigner matrices.

\begin{algorithm}[tb]
	\caption{Hypothesis test}
	\label{alg:ht}
	
	\begin{algorithmic}
		\STATE \textbf{Data}: 	$M_{ij}$, parameters $w_2, w_4$, $\lambda$
		\STATE $L_{\lambda} \gets$ test statistic in \eqref{eq:L_lambda}, \hskip5pt $m_{\lambda} \gets$ critical value in \eqref{eq:m_lambda}
		
		\IF{$L_{\lambda} \leq m_{\lambda}$} \STATE{ {\textbf{Accept}} $\bsH_1$ } 
		\ELSE \STATE{ \textbf{Accept} $\bsH_2$ }
		\ENDIF
		
	\end{algorithmic}
\end{algorithm}

\begin{thm} \label{thm:test}
	The error of the test in algorithm \ref{alg:ht} converges to
	\[
	\erfc \left( \frac{k_2 - k_1}{4} \sqrt{-\log (1-\lambda) + \left( \frac{2}{w_2} - 1 \right) \lambda + \left( \frac{1}{w_4-1} - \frac{1}{2} \right) \lambda^2} \right).
	\]
\end{thm}

\begin{proof}
	Theorem \ref{thm:test} is a direct consequence of Theorem \ref{thm:main-weak}. (Similar calculation was done in Section 3 of \cite{AlaouiJordan2018} and the proof of Theorem 2 of \cite{Chung-Lee2019}.)
\end{proof}

\begin{rem}
	In case $w_4 =3$, we obtain
	\beq \label{eq:gaussian_error}
	\lim_{N \to \infty} \err(\lambda) = \erfc \left( \frac{k_2-k_1}{4} \sqrt{-\log (1-\lambda) + \left( \frac{2}{w_2} - 1 \right) \lambda} \right).
	\eeq	
	In Theorem \ref{thm:LR_err}, we prove the error of the optimal LR test coincides with the limiting error in \eqref{eq:gaussian_error}. The limiting error in Theorem \ref{thm:LR_err} is maximal. Thus, considering that we have no information on the prior, our test is optimal when the noise is Gaussian and SNR $\lambda$ is small.
\end{rem}

\subsection{LSS-based test with entrywise transformation} \label{sec:entrywise}

For a rank-$1$ spiked matrix, if we apply a function on each entry, then the transformed matrix is approximately a rank-$1$ spiked matrix with different SNR. The function can be optimized so that the effective SNR of the transformed matrix is maximized, and it was shown that the PCA (\cite{Perry2018}) and an LSS-based test (\cite{Chung-Lee2019}) is improved with such a transformation. In this subsection, we show that the test in algorithm \ref{alg:ht} can also be improved by applying the same entrywise transformation to the data matrix as in \cite{Perry2018,Chung-Lee2019}.

We use the following technical assumptions.

\begin{assump} \label{assump:entry}
	For the spike $\bsx$, we assume that $\| \bsx \|_{\infty} \leq N^{-c}$ for some $c > \frac{3}{8}$. 
	
	For the noise, let $\caP$ and $\caP_d$ be the distributions of the normalized off-diagonal entries $\sqrt{N} H_{ij}$ and the normalized diagonal entries $\sqrt{N} H_{ii}$, respectively. We assume the following:
	\begin{enumerate}
		\item The density function $g$ of $\caP$ is smooth, positive everywhere, and symmetric (about 0).
		\item The function $h = -g'/g$ and its all derivatives are polynomially bounded in the sense that $|h^{(\ell)}(w)| \leq C_{\ell} |w|^{C_{\ell}}$ for some constant $C_{\ell}$ depending only on $\ell$.
		\item The density function $g_d$ of $\caP_d$ satisfies the assumptions 1 and 2.
	\end{enumerate}  
\end{assump}


Set $h = -g'/g$ and $h_d = -g_d'/g_d$. For a rank-$k$ spiked Wigner matrix $M$ that satisfies Assumption \ref{assump:entry}, we define a matrix $\tM$ by
\beq \label{eq:tM1}
\tM_{ij} = \frac{1}{\sqrt{\fh N}} h(\sqrt{N} M_{ij}) \quad (i \neq j), \qquad 
\tM_{ii} = \sqrt{\frac{w_2}{\fh_d N}} h_d \left(\sqrt{\frac{N}{w_2}} M_{ii} \right),
\eeq
where
\[
\fh = \int_{-\infty}^{\infty} \frac{g'(w)^2}{g(w)} \dd w, \qquad \fh_d = \int_{-\infty}^{\infty} \frac{g_d'(w)^2}{g_d(w)} \dd w.
\]
The transformation has an effect of changing the SNR from $\lambda$ to $\lambda \fh$, which is an improvement when the noise is non-Gaussian since $\fh \geq 1$ and the equality holds if and only if $\caP$ is a standard Gaussian. For more detail, we refer to \cite{Perry2018,Chung-Lee2019}.

Following \cite{Chung-Lee2019}, we consider a test statistic
\beq \begin{split} \label{eq:wt L_lambda}
	\wt L_{\lambda} &:= - \log \det \left( (1+\lambda\fh)I - \sqrt{\lambda\fh} \tM \right) + \frac{\lambda\fh}{2} N \\
	&\qquad + \sqrt{\lambda} \left( \frac{2\sqrt{\fh_d}}{w_2} - \sqrt{\fh} \right) \Tr \tM 
	+ \lambda \left( \frac{\gh}{\wt{w_4}-1} - \frac{\fh}{2} \right) (\Tr \tM^2 - N),
\end{split} 
\eeq
where 
\[
\gh = \frac{1}{2\fh} \int_{-\infty}^{\infty} \frac{g'(w)^2 g''(w)}{g(w)^2} \dd w,
\quad \wt{w_4} = \frac{1}{(\fh)^2} \int_{-\infty}^{\infty} \frac{(g'(w))^4}{(g(w))^3} \dd w.
\]
We then have the following CLT result for $\wt L_{\lambda}$ that generalizes Theorem 3 of \cite{Chung-Lee2019}.

\begin{thm} \label{thm:trans_main}
	For a rank-$k$ spiked Wigner matrix $M$ with $\lambda \fh < 1$ satisfying Assumption \ref{assump:entry}.
	\beq
	\wt L_{\lambda} \Rightarrow \caN(\wt m_k, \wt V_0),
	\eeq
	where the mean and the variance are given by
	\beq \begin{split}
		\wt m_k &= - \frac{1}{2} \log(1-\lambda\fh) + \left( \frac{(w_2 -1)\gh}{\wt w_4 -1} - \frac{\fh}{2} \right) \lambda + \frac{\wt w_4 -3}{4} (\lambda\fh)^2 \\
		&\qquad + k \left[ -\log(1- \lambda\fh) + \left( \frac{2\fh_d}{w_2} - \fh \right) \lambda + \left( \frac{(\gh)^2}{\wt w_4-1} - \frac{(\fh)^2}{2}\right) \lambda^2 \right],
	\end{split} \eeq
	\beq \begin{split}
		\wt V_0 = -2 \log(1- \lambda\fh) + \left( \frac{4\fh_d}{w_2} - 2\fh \right) \lambda + \left( \frac{2(\gh)^2}{\wt w_4-1} - (\fh)^2 \right) \lambda^2.
	\end{split} \eeq
\end{thm}

\begin{proof}
	Theorem \ref{thm:trans_main} directly follows from Theorem \ref{thm:trans_CLT} in Section \ref{sec:LSS}.
\end{proof}

Analogous to Algorithm \ref{alg:ht} and also Algorithm 2 of \cite{Chung-Lee2019}, we propose a test described in Algorithm \ref{alg:htet} where we compute $\wt L_{\Lambda}$ and compare it with the critical value
\beq \begin{split} \label{eq:wt m_lambda}
	\wt m_{\lambda} := (\wt m_{k_1} + \wt m_{k_2})/2.
\end{split} \eeq

\begin{thm} \label{thm:trans_test}
	The error of the test in Algorithm \ref{alg:htet} converges to
	\[
	\erfc \left( \frac{k_2-k_1}{4} \sqrt{-\log (1-\lambda\fh) + \left( \frac{2\fh_d}{w_2} - \fh \right) \lambda + \left( \frac{(\gh)^2}{\wt w_4-1} - \frac{(\fh)^2}{2} \right) \lambda^2} \right).	
	\]
\end{thm}

\begin{proof}
	Theorem \ref{thm:test} is a direct consequence of Theorem \ref{thm:trans_CLT}.
\end{proof}

In Appendix \ref{sec:ex}, we consider spiked Wigner matrices with non-Gaussian noise and show both theoretically
and numerically that the error from Algorithm \ref{alg:htet} is lower than that of Algorithm \ref{alg:ht}.

\begin{algorithm}[h]
	\caption{Hypothesis test with the entrywise transformation}
	\label{alg:htet}
	
	\begin{algorithmic}
		\STATE \textbf{Data}: $M_{ij}$, parameters $w_2, w_4$, $\lambda$, densities $g, g_d$
		\STATE $\tM \gets$ transformed matrix in \eqref{eq:tM1}, \hskip5pt $\wt L_{\lambda} \gets$ test statistic in \eqref{eq:wt L_lambda}, \hskip5pt $\wt m_{\lambda} \gets$ critical value in \eqref{eq:wt m_lambda}
		
		\IF{$\wt L_{\lambda} \leq \wt m_{\lambda}$} \STATE{ {\textbf{Accept}} $\bsH_1$ }
		\ELSE \STATE{ {\textbf{Accept}} $\bsH_2$ }
		\ENDIF
		
	\end{algorithmic}
\end{algorithm}

\subsection{Rank estimation} \label{subsec:adaptive}

The test in Algorithm \ref{alg:ht} requires prior knowledge about $k_1$ and $k_2$, the possible ranks of the planted spike. In this section, we adapt the idea to estimate the rank of the signal when there is no prior information on the rank $k$. Recall that the test statistic $L_{\lambda}$ defined in \eqref{eq:L_lambda} does not depend on the rank of the matrix. As proved in Theorem \ref{thm:main-weak}, the test statistic $L_{\lambda}$ converges to a Gaussian random variable with mean $m_k$ and the variance $V_0$, where $m_k$ is equi-distributed with respect to $k$ and $V_0$ does not depend on $k$. It is then natural to set the best candidate for $k$, which we call $\kappa$, be the minimizer of the distance $|L_{\lambda} - m_k|$. This procedure is equivalent to find the nearest nonnegative integer of the value
\beq \label{eq:adaptive_value}
\kappa' := \frac{L_{\lambda} - m_0}{-\log(1-\lambda) + \left( \frac{2}{w_2} - 1 \right) \lambda + \left( \frac{1}{w_4-1} - \frac{1}{2} \right) \lambda^2}
\eeq
rounding half down.

We describe the test in Algorithm \ref{alg:at}; its probability of error converges to
\beq \begin{split} \label{eq:rank_k_error}
	& \p(k=0) \cdot \p \left( Z > \frac{\sqrt{V_0}}{4} \right) + \sum_{i=1}^{\infty} \p(k=i) \cdot \p \left( |Z| > \frac{\sqrt{V_0}}{4} \right) \\
	&= \left( 1 - \frac{\p(k=0)}{2} \right) \cdot\erfc \left( \frac{1}{4} \sqrt{-\log (1-\lambda) + \left( \frac{2}{w_2} - 1 \right) \lambda + \left( \frac{1}{w_4-1} - \frac{1}{2} \right) \lambda^2} \right),
\end{split} \eeq
where $Z$ is a standard Gaussian random variable. Note that it depends only on $\p(k=0)$.

\begin{algorithm}[h]
	\caption{Rank estimation}
	\label{alg:at}
	
	\begin{algorithmic}
		\STATE \textbf{Data}: $M_{ij}$, parameters $w_2, w_4$, $\lambda$
		\STATE $L_{\lambda} \gets$ test statistic in \eqref{eq:L_lambda}, \hskip5pt $m_0 \gets$ mean in \eqref{eq:m_lambda}, \hskip5pt $m_1 \gets$ mean in \eqref{eq:m_k} with $k=1$
		\STATE $\kappa' \gets$ value in \eqref{eq:adaptive_value}
		
		\IF{$L_{\lambda} \leq (m_0 + m_1)/2$} \STATE{ {\textbf{Set}} $\kappa=0$ }
		\ELSE \STATE{ {\textbf{Set}} $\kappa = \lceil \kappa' - 0.5 \rceil$ }
		\ENDIF
		
	\end{algorithmic}
\end{algorithm}

The error can be lowered if the range of $k$ is known a priori. See Appendix \ref{sec:ex}.

\section{LR test for spiked Gaussian Wigner matrices} \label{sec:Gaussian}

We next compare the limiting error of the proposed test in Theorem \ref{thm:test} and that of the LR test. We consider the fluctuation of the LR of the spiked Gaussian Wigner model defined in Definition \ref{def:Gaussian}.

\begin{defn}[Likelihood ratio] 
	For a data matrix $Y$ in Definition \ref{def:Gaussian}, the likelihood ratio (or the Radon--Nikodym derivative) of $\p_2$ with respect to $\p_1$ is 
	\[
	\caL(Y; k_1, k_2):=\frac{\dd\p_2}{\dd\p_1}.
	\]
\end{defn}

\subsection{Gaussian convergence of the log-LR for spiked Gaussian Wigner matrices}

Let $X^{[\ell]}_i = (x^{[\ell]}_i(1), x^{[\ell]}_i(2), \dots, x^{[\ell]}_i(k_\ell))$ be the $i$-th row vector of $X^{[\ell]}$ for $\ell=1,2$. Note that 
\[
X^{[1]}_i\sim\caP^{\otimes k_1} =: \caP_{0,1}, \qquad X^{[2]}_i \sim \caP^{\otimes k_2} =: \caP_{0,2} \qquad (i=1, 2, \dots, N).
\] 
Similarly, we also let $X_i^\ast$ be the $i$-th row vector of $X^\ast$, and	
\[
X_i^\ast=(x_i^\ast(1),\dots,x_i^\ast(k_2)) \sim \caP_{0,2}.
\]
Conditioning on $X^*$, from the Gaussianity of $W$, we first observe that the posterior distribution of $X$ for given data $Y$ is
\[
\dd\bbP_\ell(X|Y)=\frac{e^{-H^{k_\ell}(X^{[\ell]})} \dd\caP_{0,\ell}^{\otimes N}(X^{[\ell]})}{\int e^{-H^{k_\ell}(X^{[\ell]})} \dd\caP_{0,\ell}^{\otimes N}(X^{[\ell]})}
\]
where the Hamiltonian $H^{k_\ell}(X^{[\ell]})$ is given by
\[ 
\begin{split} 
-H^{k_{\ell}}(X^{[\ell]}) &= \sum_{i<j}^N \sum_{n=1}^{k_\ell} \left[ \sqrt{\frac{\lambda}{N}} Y_{ij} \,x^{[\ell]}_i(n) \,x^{[\ell]}_j(n)- \frac{\lambda}{2N} x^{[\ell]}_i(n) \,x^{[\ell]}_j(n)\sum_{m=1}^{k_\ell}x^{[\ell]}_i(m) \,x^{[\ell]}_j(m) \right] \\
&\quad + \frac{1}{w_2} \sum_{i=1}^{N} \sum_{n=1}^{k_\ell} \left[ \sqrt{\frac{\lambda}{N}} Y_{ii}\, x^{[\ell]}_i(n)^2 - \frac{\lambda}{2N} x^{[\ell]}_i(n)^2\sum_{m=1}^{k_\ell}x^{[\ell]}_i(m)^2 \right]
\end{split} 
\]
for $\ell=1,2$. For convenience, we let 
\[
\caL(Y;k_\ell)=\int e^{-H^{k_\ell}(X^{[\ell]})} \dd\caP_{0,\ell}^{\otimes N}(X^{[\ell]}).
\]
From Bayes' theorem, it is natural to define the LR between two hypotheses by
\beq
\caL(Y; k_1, k_2)=\caL(Y;k_2)/\caL(Y;k_1).
\eeq

The log-LR converges to a Gaussian as in the following theorem.

\begin{thm}\label{thm:main}
	Assume that the prior $\caP$ is centered, has unit variance and bounded support. Then, there exists $\lambda_0 \equiv \lambda_0 (k_1, k_2,\caP) \in (0, 1]$ such that for $\lambda<\lambda_0$ 
	\[
	\log\caL(Y; k_1, k_2)\Rightarrow\mathcal{N}(-\mu, 2\mu) \qquad \text{ under } \bsH_1,
	\]
	\[
	\log\caL(Y; k_1, k_2)\Rightarrow\mathcal{N}(\mu, 2\mu) \qquad \text{ under } \bsH_2,
	\]
	where
	\beq
	\mu \equiv \mu(k_1, k_2) = \frac{(k_1-k_2)^2}{4} \left(-\log (1-\lambda) + \left( \frac{2}{w_2} - 1 \right) \lambda\right).
	\eeq
\end{thm}
We give a sketch of the proof of Theorem \ref{thm:main} in Section \ref{sec:LR}; see Appendix \ref{app:thm} for the detail of the proof.
With Theorem \ref{thm:main} and Le Cam's first lemma, we obtain the following corollary by a contiguity argument.

\begin{cor}\label{cor:contiguity}
	Under the assumptions of Theorem \ref{thm:main}, $\p_1$ and $\p_2$ are mutually contiguous.
\end{cor}

In the LR test, we accept $\bsH_1$ if $\caL(Y; k_1, k_2) \leq 1$ and accept $\bsH_2$ if $\caL(Y; k_1, k_2) > 1$. The error of such a test is
\beq
\err^*(k_1, k_2)=\p_1(\caL(Y; k_1, k_2)\leq1)+\p_2(\caL(Y; k_1, k_2)>1).
\eeq 
In the next theorem, we compute the limiting error of the LR test, which coincides with the error of the proposed test in Theorem \ref{thm:test}.

\begin{thm}\label{thm:LR_err}
	Under the assumptions of Theorem \ref{thm:main}, if $\lambda < \lambda_0$ then
	\beq \label{eq:error_limit}
	\lim_{N\to\infty}\err^*(k_1, k_2) = \erfc\left(\frac{k_2-k_1}{4} \sqrt{-\log (1-\lambda) + \left( \frac{2}{w_2} - 1 \right) \lambda}\right)
	\eeq
	where $\erfc(\cdot)$ is the complementary error function defined as $\erfc(x)=\frac{2}{\sqrt{\pi}}\int_{x}^{\infty}e^{-t^2}\dd t$.
\end{thm}

\begin{proof}
	Theorem \ref{thm:LR_err} is a direct consequence of Theorem \ref{thm:main}.
\end{proof}


\subsection{Sketch of the proof of Theorem \ref{thm:main}} \label{sec:LR}

First, note that it suffices to prove the statement under any of the hypotheses since fluctuation under the other is derived easily as a consequence of Le Cam's third Lemma (see, e.g., Theorem 6.6 in \cite{Vaart1998}.) For simplicity, we assume $\bsH_2$.

Our main strategy for the proof of Theorem \ref{thm:main} is to analyze the limiting behavior of the characteristic function $\phi_N(\lambda)$ of the log-LR, defined as 
\beq\label{char_log-LR}
\phi_N(\lambda)=\E_{\p_2}\left[e^{\ii s\log\caL(Y; k_1, k_2)}\right]
\eeq
for a fixed $s \in \R$. 

In the case $w_2 = \infty$, differentiating $\phi_N$ and applying Gaussian integration by parts, we find that $\phi_N$ is asymptotically the solution of the initial value problem of the ODE
\beq \label{eq:phi'}
\phi_N'(\lambda) = \frac{\ii s-s^2}{4} \cdot \frac{(k_1-k_2)^2 \lambda}{1-\lambda} \phi_N(\lambda) + O(N^{-\frac{1}{2}}).
\eeq
with the initial value $\phi_N(0) = 1$.
(See also Lemma 8 and Proposition 9 of \cite{AlaouiJordan2018}.) Since $\phi_N(0) = 1$, integrating \eqref{eq:phi'} with respect to $\lambda$, for any $\lambda<\lambda_0(k_1,k_2)$ and $s\in\R$ 
\beq
|\phi_N(\lambda)-e^{(\ii s-s^2)\mu}| = O(N^{-\frac{1}{2}}),
\eeq 
where $\mu=\frac{(k_2-k_1)^2}{4}\left(-\log(1-\lambda)-\lambda\right)$. The desired result for $w_2 = \infty$ now directly follows.

In the case $w_2 < \infty$, we need to add the contribution from the diagonal term. Following the cavity computation in \cite{AlaouiJordan2018}, we obtain that $\phi_N$ satisfies the following deformed ODE 
\beq \label{eq:phi'2}
\phi_N'(\lambda) = \frac{\ii s-s^2}{4} \cdot \frac{(k_1-k_2)^2 \lambda}{1-\lambda}  \phi_N(\lambda) + \frac{\ii s - s^2}{2w_2} \cdot  (k_1-k_2)^2 \phi_N(\lambda) + O(N^{-\frac{1}{2}}),
\eeq
with the same initial value in case without diagonal elements and we obtain the desired result by integrating \eqref{eq:phi'2} with respect to $\lambda$. See Appendix \ref{app:thm} for more detail on the derivation of \eqref{eq:phi'} and \eqref{eq:phi'2}.

\section{Central Limit Theorems for Spiked Wigner Matrices} \label{sec:LSS}

In this section, we collect our results on general CLTs for the LSS of spiked Wigner matrices. To precisely define the statements, we introduce the Chebyshev polynomials of the first kind.

\begin{defn}[Chebyshev polynomial]
	The $n$-th Chebyshev polynomial (of the first kind) $T_n$ is a degree $n$ polynomial defined by $T_0(x) = 1$, $T_1(x) = x$, and
	\[
	T_{n+1}(x) = 2x T_n(x) - T_{n-1}(x).
	\]
\end{defn}

We first state a CLT for the LSS that generalizes Theorem 5 of \cite{Chung-Lee2019}.
\begin{thm} \label{thm:CLT}
	Assume the conditions in Theorem \ref{thm:main-weak}. Then, for any function $f$ analytic on an open interval containing $[-2, 2]$,
	\[
	\left( \sum_{i=1}^N f(\mu_i) - N \int_{-2}^2 \frac{\sqrt{4-z^2}}{2\pi} f(z) \, \dd z \right) \Rightarrow \caN\left(m_k(f), V_0(f)\right)\,.
	\]
	The mean and the variance of the limiting Gaussian distribution are given by
	\[ \begin{split} 
	m_k(f) = \frac{1}{4} \left( f(2) + f(-2) \right) -\frac{1}{2} \tau_0(f) + (w_2 -2) \tau_2(f) 
	+ (w_4-3) \tau_4(f) + k \sum_{\ell=1}^{\infty} \sqrt{\lambda^{\ell}} \tau_{\ell}(f),
	\end{split} 
	\]
	\[ \begin{split} 
	V_0(f) = (w_2-2) \tau_1(f)^2 + 2(w_4-3) \tau_2(f)^2 
	+ 2\sum_{\ell=1}^{\infty} \ell \tau_{\ell}(f)^2\,,
	\end{split} 
	\]
	where we let
	\[
	\tau_{\ell}(f) = \frac{1}{\pi} \int_{-2}^2 T_{\ell} \left( \frac{x}{2} \right) \frac{f(x)}{\sqrt{4-x^2}} \dd x.
	\]
\end{thm}
We will give a proof of Theorem \ref{thm:CLT} in Appendix \ref{app:CLT}.

The next result shows that the proposed test in Algorithm \ref{alg:ht} achieves the lowest error among all tests based on LSS.
\begin{thm} \label{thm:optimize}
	Assume the conditions in Theorem \ref{thm:CLT}. If $w_2 > 0$ and $w_4 > 1$, then
	\beq \label{eq:upper_bound}
	\left|\frac{m_{k_2}(f) - m_{k_1}(f)}{\sqrt{V_0(f)}} \right| \leq \left| \frac{m_{k_2}-m_{k_1}}{\sqrt{V_0}} \right|.
	\eeq
	The equality holds if and only if $f = C_1 \varphi_{\lambda} + C_2$ for some constants $C_1$ and $C_2$ where
	\[
	\varphi_{\lambda}(x) := \log \left( \frac{1}{1-\sqrt{\lambda}x + \lambda} \right) + \sqrt{\lambda} \left( \frac{2}{w_2} - 1 \right) x 
	+ \lambda^2 \left( \frac{1}{w_4-1} - \frac{1}{2} \right) x^2.
	\]
\end{thm}

\begin{proof}
	The theorem easily follows from the proof of Theorem 6 in \cite{Chung-Lee2019} with applying Theorem \ref{thm:CLT} instead of Theorem 5 in \cite{Chung-Lee2019}.
\end{proof}

With the entrywise transformation in Section \ref{sec:entrywise}, we have the following changes in Theorems \ref{thm:CLT} and \ref{thm:optimize}.  
\begin{thm} \label{thm:trans_CLT}
	For a rank-$k$ spiked Wigner matrix $M$ with $\lambda \fh < 1$ satisfying Assumption \ref{assump:entry} and for any function $f$ analytic on an open interval containing $[-2, 2]$,
	\[ \begin{split}
	\left( \sum_{i=1}^N f(\wt\mu_i) - N \int_{-2}^2 \frac{\sqrt{4-z^2}}{2\pi} f(z) \, \dd z \right)
	\Rightarrow \caN(\wt m_k(f), \wt V_0(f))\,.
	\end{split} 
	\]
	The mean and the variance of the limiting Gaussian distribution are given by
	\beq \begin{split} \label{eq:mean_tM}
		\wt m_k(f) &= \frac{1}{4} \left( f(2) + f(-2) \right) -\frac{1}{2} \tau_0(f) + k \sqrt{\lambda\fh_d} \tau_1 (f) 
		+ (w_2 -2 + k \lambda\gh) \tau_2(f)\\
		&\qquad + (\wt{w_4}-3) \tau_4(f) + k \sum_{\ell=3}^{\infty} \sqrt{(\lambda\fh)^\ell} \tau_{\ell}(f),
	\end{split} 
	\eeq
	\[	
	V_{\tM}(f) = V_M(f) = (w_2-2) \tau_1(f)^2 + 2(\wt{w_4}-3) \tau_2(f)^2 + 2\sum_{\ell=1}^{\infty} \ell \tau_{\ell}(f)^2.
	\]
\end{thm}
We will also prove Theorem \ref{thm:trans_CLT} in Appendix \ref{app:CLT}.
\begin{thm} \label{thm:trans_optimize}
	Assume the conditions in Theorem \ref{thm:trans_CLT}. If $w_2 > 0$ and $\wt{w_4} > 1$, then
	\beq \label{eq:trans_upper_bound}
	\left| \frac{\wt m_{k_2}(f) - \wt m_{k_1}(f)}{\sqrt{\wt V_0(f)}} \right| \leq \left|\frac{ \wt m_{k_2} - \wt m_{k_1}}{ \sqrt{\wt V_0}}\right|.
	\eeq
	Here, the equality holds if and only if $f(x) = C_1 \wt\varphi_{\lambda}(x) + C_2$ for some constants $C_1$ and $C_2$ where
	\[ \begin{split}
	\wt\varphi_{\lambda}(x) := \log \left( \frac{1}{1-\sqrt{\lambda\fh} x + \lambda\fh} \right) 
	+ \sqrt{\lambda} \left( \frac{2\sqrt{\fh_d}}{w_2} - \sqrt{\fh} \right) x + \lambda \left( \frac{\gh}{\wt{w_4}-1} - \frac{\fh}{2} \right) x^2.
	\end{split} 
	\]
\end{thm}

\section{Conclusion and Future Works} \label{sec:summary}

In this paper, we considered the weak detection of the spiked Wigner model with general ranks. We proposed a hypothesis test based on the central limit theorem for the linear spectral statistics of the data matrix and introduced a test for rank estimation that do not require any prior information on the rank of the signal. It was shown that the error of the proposed hypothesis test matches the error of the likelihood ratio test in case the noise is Gaussian and the signal-to-noise ratio is small. With the knowledge on the density of the noise, the test was further improved by applying an entrywise transformation.

We believe it is possible to consider the detection problem in the spiked rectangular model, where the data matrix is not necessarily symmetric nor even square. We believe that the hypothesis test proposed in this paper can be extended to the spiked rectangular model, where we may form sample covariance matrices (Gram matrices) and apply the central limit theorem for the linear spectral statistics. This will be discussed in our future works.

\subsubsection*{Acknowledgments} 
The work of J. H. Jung and J. O. Lee was partially supported by National Research Foundation of Korea under grant number NRF-2019R1A5A1028324. The work of H. W. Chung was partially supported by National Research Foundation of Korea under grant number 2017R1E1A1A01076340 and by the Ministry of Science and ICT, Korea, under an ITRC Program, IITP-2019-2018-0-01402.

\bibliographystyle{abbrv}
\bibliography{references}

\appendix

\section{Examples and Simulations} \label{sec:ex}

In Appendix \ref{sec:ex}, we numerically check the errors of the proposed tests in Algorithms \ref{alg:ht} and \ref{alg:htet} and the test for rank estimation in Algorithm \ref{alg:at} under various settings.

\subsection{Spiked Gaussian Wigner model} \label{subsec:GOE}

We consider the simplest case of the spiked Gaussian Wigner model where $w_2 = 2$ (i.e., $H$ is a GOE matrix) and the signal $\bsx(m)  = (x_1(m), x_2(m), \dots, x_N(m))$ where $\sqrt{N} x_i(m)$'s are i.i.d. Rademacher random variable. Note that the parameters $w_2 = 2$ and $w_4 = 3$. 

In the numerical simulation done in Matlab, we generated 10,000 independent samples of the $256 \times 256$ data matrix $M$, where we fix $k_1=1$ (under $\bsH_1$) and vary $k_2$ from $2$ to $5$ (under $\bsH_2$), with the SNR $\lambda$ varying from $0$ to $0.7$. To apply Algorithm \ref{alg:ht}, we compute
\beq \begin{split}
	L_{\lambda} = -\log \det \big( (1+\lambda)I - \sqrt{\lambda} M \big) + \frac{\lambda N}{2}.
\end{split} 
\eeq
We accept $\bsH_1$ if 
\[
L_{\lambda} \leq \frac{m_{k_1}+m_{k_2}}{2} = -\frac{k_2 + 2}{2} \log(1-\lambda)
\]
and accept $\bsH_2$ otherwise. The (theoretical) limiting error of the test is
\beq \label{eq:limit_error_1a}
\erfc \left( \frac{k_2 - 1}{4} \sqrt{-\log (1-\lambda)} \right).
\eeq

In Figure \ref{fig:LSS_GOE}, we compare the error from the numerical simulation and the theoretical error of the proposed algorithm, which show that the numerical errors of the test closely match the theoretical errors. 

\begin{figure}[h]
	\vskip 0.2in
	\begin{center}
		\centerline{\includegraphics[width=250pt]{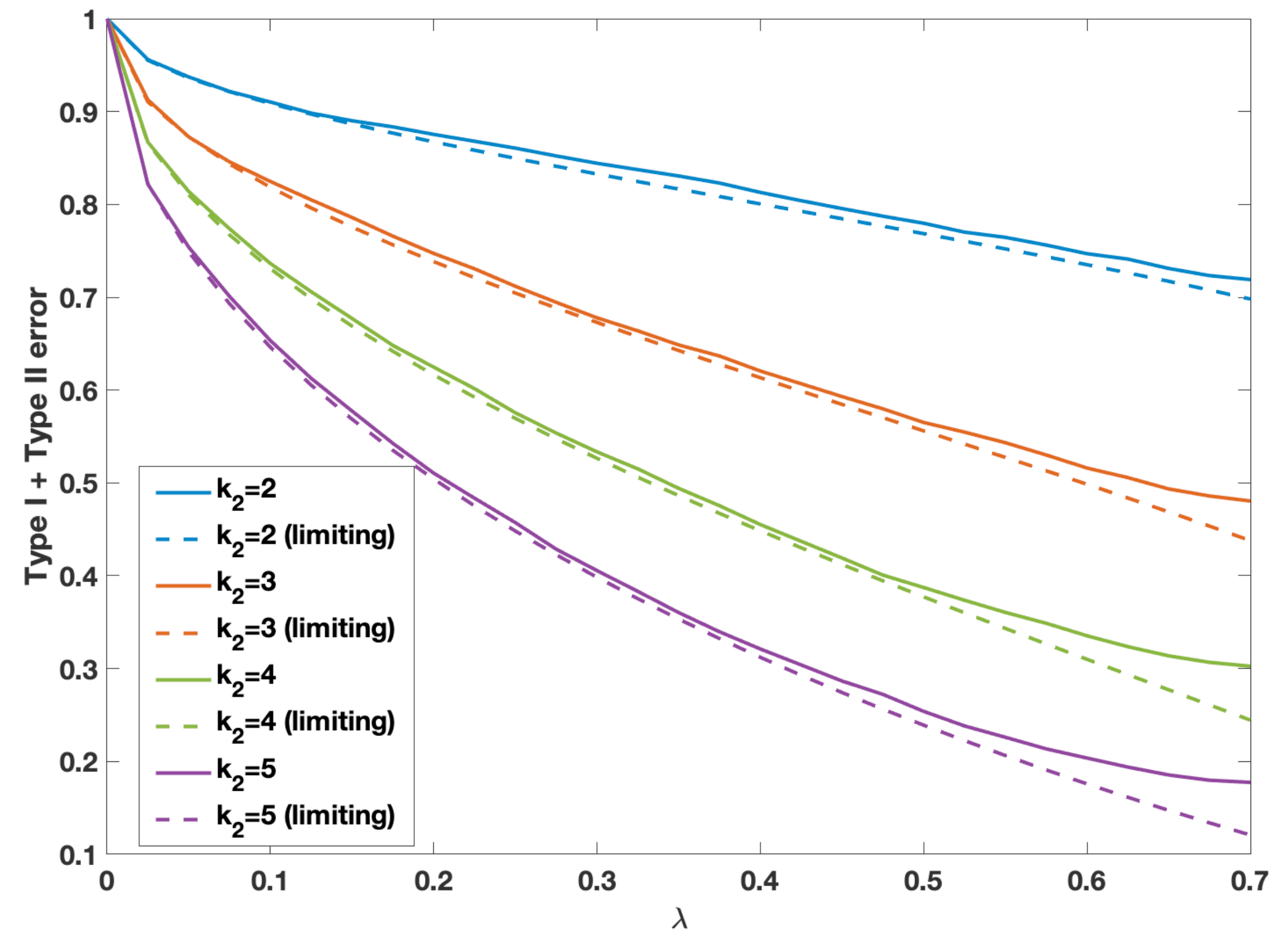}}
		\caption{The errors from the simulation with Algorithm \ref{alg:ht} (solid) versus the limiting errors \eqref{eq:limit_error_1a} (dashed) for the setting in Section \ref{subsec:GOE} with $k_2 = 2, 3, 4, 5$.}
		\label{fig:LSS_GOE}
	\end{center}
	\vskip -0.2in
\end{figure}

\subsection{Spiked Wigner model} \label{subsec:sech}

We next consider a spiked Wigner model with non-Gaussian noise, where the density function of the noise matrix is given by
\[
g(x) = g_d(x) = \frac{1}{2 \cosh (\pi x/2)} = \frac{1}{e^{\pi x/2} + e^{-\pi x/2}}.
\]
We sample $W_{ij} = W_{ji}$ from the density $g$ and let $H_{ij} = W_{ij}/\sqrt{N}$. We again let the signal $\bsx(m)  = (x_1(m), x_2(m), \dots, x_N(m))$ where $\sqrt{N} x_i(m)$'s are i.i.d. Rademacher random variable. Note that the parameters $w_2 = 1$ and $w_4 = 5$. We again perform the numerical simulation 10,000 samples of the $256 \times 256$ data matrix $M$ with the SNR $\lambda$ varying from $0$ to $0.6$, where we fix $k_1=1$ (under $\bsH_1$) and $k_2=3$ (under $\bsH_2$).

In Algorithm \ref{alg:ht}, we compute
\beq \begin{split}
	L_{\lambda} = -\log \det \big( (1+\lambda)I - \sqrt{\lambda} M \big) + \frac{\lambda N}{2} 
	+ \sqrt{\lambda} \Tr M - \frac{\lambda}{4} (\Tr M^2 - N).
\end{split} 
\eeq
We accept $\bsH_1$ if 
\[
L_{\lambda} \leq \frac{m_{k_1}+m_{k_2}}{2} = -\frac{k_2 + 2}{2} \log(1-\lambda) + \frac{k_2 \lambda}{2} - \frac{(k_2 -3) \lambda^2}{8}
\]
and accept $\bsH_2$ otherwise. The (theoretical) limiting error of the test is
\beq \label{eq:limit_error_1}
\erfc \left( \frac{k_2 - 1}{4} \sqrt{-\log (1-\lambda)  + \lambda - \frac{\lambda^2}{4}} \right).
\eeq

We can further improve the test by introducing the entrywise transformation given by
\[
h(x) = -\frac{g'(x)}{g(x)} = \frac{\pi}{2} \tanh \frac{\pi x}{2}.
\]
The Fisher information $\fh = \frac{\pi^2}{8}$, which is larger than $1$. We thus construct a transformed matrix $\wt M$ by
\[
\wt M_{ij} = \frac{2\sqrt{2}}{\pi \sqrt{N}} h(\sqrt{N} M_{ij}) = \sqrt{\frac{2}{N}} \tanh \left( \frac{\pi \sqrt{N}}{2} M_{ij} \right).
\]
If $\lambda > \frac{1}{\fh} = \frac{8}{\pi^2}$, we can apply PCA for strong detection of the signal. If $\lambda < \frac{8}{\pi^2}$, applying Algorithm \ref{alg:htet}, we compute
\[ \begin{split}
\wt L_{\lambda} = -\log \det \left( \left(1+\frac{\pi^2 \lambda}{8}\right)I - \sqrt{\frac{\pi^2 \lambda}{8}} \tM \right) + \frac{\pi^2 \lambda N}{16} 
+ \frac{\pi \sqrt{\lambda}}{2 \sqrt{2}} \Tr \tM + \frac{\pi^2 \lambda}{16} (\Tr \tM^2 - N).
\end{split} 
\]
(Here, $\fh = \fh_d = \frac{\pi^2}{8}$, $\gh = \frac{\pi^2}{16}$, and $\wt w_4 = \frac{3}{2}$.) We accept $\bsH_1$ if 
\[
\wt L_{\lambda} \leq -\frac{k_2 +2}{2} \log \left(1-\frac{\pi^2 \lambda}{8} \right) + \frac{k_2 \pi^2 \lambda}{16} - \frac{3 \pi^4 \lambda^2}{512}
\]
and accept $\bsH_2$ otherwise. The limiting error with entrywise transformation is
\beq \label{eq:limit_error_2}
\erfc \left( \frac{k_2 - 1}{4} \sqrt{ -\log \left(1-\frac{\pi^2 \lambda}{8}\right) + \frac{\pi^2 \lambda}{8}} \right).
\eeq
Since $\erfc(\cdot)$ is a decreasing function and $\frac{\pi^2}{8} > 1$, it is immediate to see that the limiting error in \eqref{eq:limit_error_2} is strictly smaller than the limiting error in \eqref{eq:limit_error_1}.

In Figure \ref{fig:alg12}, we plot the result of the simulation with $k_2=3$, which shows that the numerical error from Algorithm \ref{alg:htet} is smaller than that of Algorithm \ref{alg:ht}; both errors closely match theoretical errors in \eqref{eq:limit_error_2} and \eqref{eq:limit_error_1}.

\begin{figure}[h!]
	\vskip 0.2in
	\begin{center}
		\centerline{\includegraphics[width=250pt]{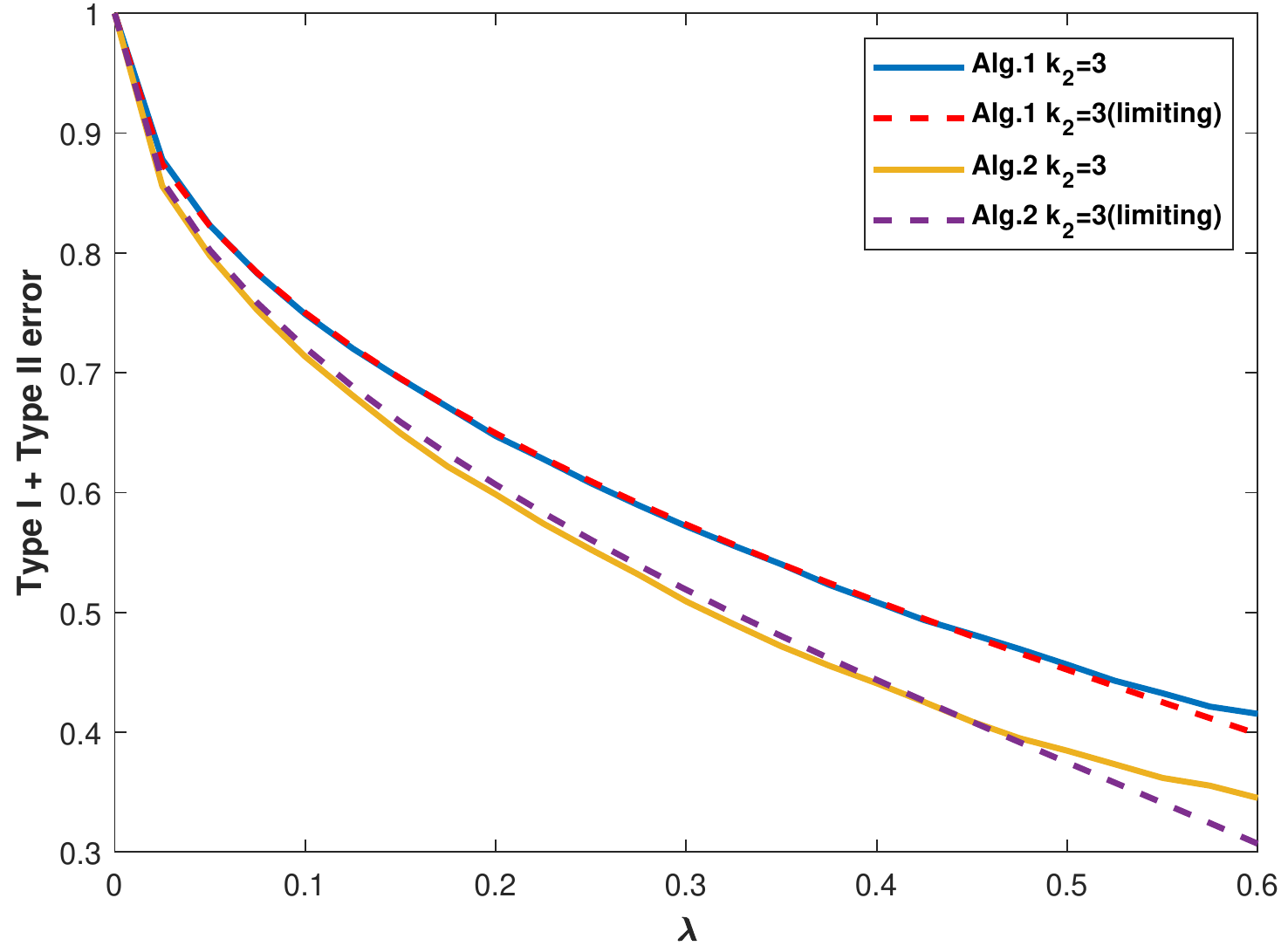}}
		\caption{The errors from the simulation with Algorithm \ref{alg:ht} (blue) and with Algorithm \ref{alg:htet} (yellow), respectively, versus the limiting errors \eqref{eq:limit_error_1} of Algorithm \ref{alg:ht} (red) and \eqref{eq:limit_error_2} of Algorithm \ref{alg:htet} (purple), respectively, for the setting in Section \ref{subsec:sech}.}
		\label{fig:alg12}
	\end{center}
	\vskip -0.2in
\end{figure}

\subsection{Rank Estimation} \label{subsec:rank}

We again consider the example in Section \ref{subsec:GOE} and apply Algorithm \ref{alg:at} to estimate the rank of the signal. We again perform the numerical simulation 20,000 samples of the $256 \times 256$ data matrix $M$ with the SNR $\lambda$ varying $0.025$ to $0.6$ and choose the rank of the signal $k$ uniformly from $0$ to $4$. Since we know that the range of the rank $k$ is $[0, 4]$, the (theoretical) limiting error in \eqref{eq:rank_k_error} changes to
\[ \begin{split}
& \p(k=0) \cdot \p \left( Z > \frac{\sqrt{V_0}}{4} \right) + \sum_{i=1}^{3} \p(k=i) \cdot \p \left( |Z| > \frac{\sqrt{V_0}}{4} \right) + \p(k=4) \cdot \p \left( Z > \frac{\sqrt{V_0}}{4} \right) \\
&= \left( 1 - \frac{\p(k=0)+\p(k=4)}{2} \right) \cdot\erfc \left( \frac{1}{4} \sqrt{-\log (1-\lambda) + \left( \frac{2}{w_2} - 1 \right) \lambda + \left( \frac{1}{w_4-1} - \frac{1}{2} \right) \lambda^2} \right).
\end{split} \]

We compute the same test statistic
\beq \begin{split}
	L_{\lambda} = -\log \det \big( (1+\lambda)I - \sqrt{\lambda} M \big) + \frac{\lambda N}{2}
\end{split} \eeq
and find the nearest nonnegative integer of the value
\beq
-\frac{L_{\lambda}}{\log (1-\lambda)} - \frac{1}{2},
\eeq
rounding half down. 
Since $\p(k=0)=\p(k=4) = 0.2$, the limiting error of the estimation is
\beq \label{eq:limit_rank}
\left(1-\frac{\p(k=0)+\p(k=4)}{2}\right) \cdot \erfc \left( \frac{1}{4} \sqrt{-\log (1-\lambda)}\right)=	0.8 \cdot \erfc \left( \frac{1}{4} \sqrt{-\log (1-\lambda)}\right).
\eeq

The result of the simulation can be found in Figure \ref{fig:alg_at}, where we compare the error from the estimation (Algorithm \ref{alg:at}) and the theoretical error in \eqref{eq:limit_rank}.

\begin{figure}[h]
	\vskip 0.2in
	\begin{center}
		\centerline{\includegraphics[width=250pt]{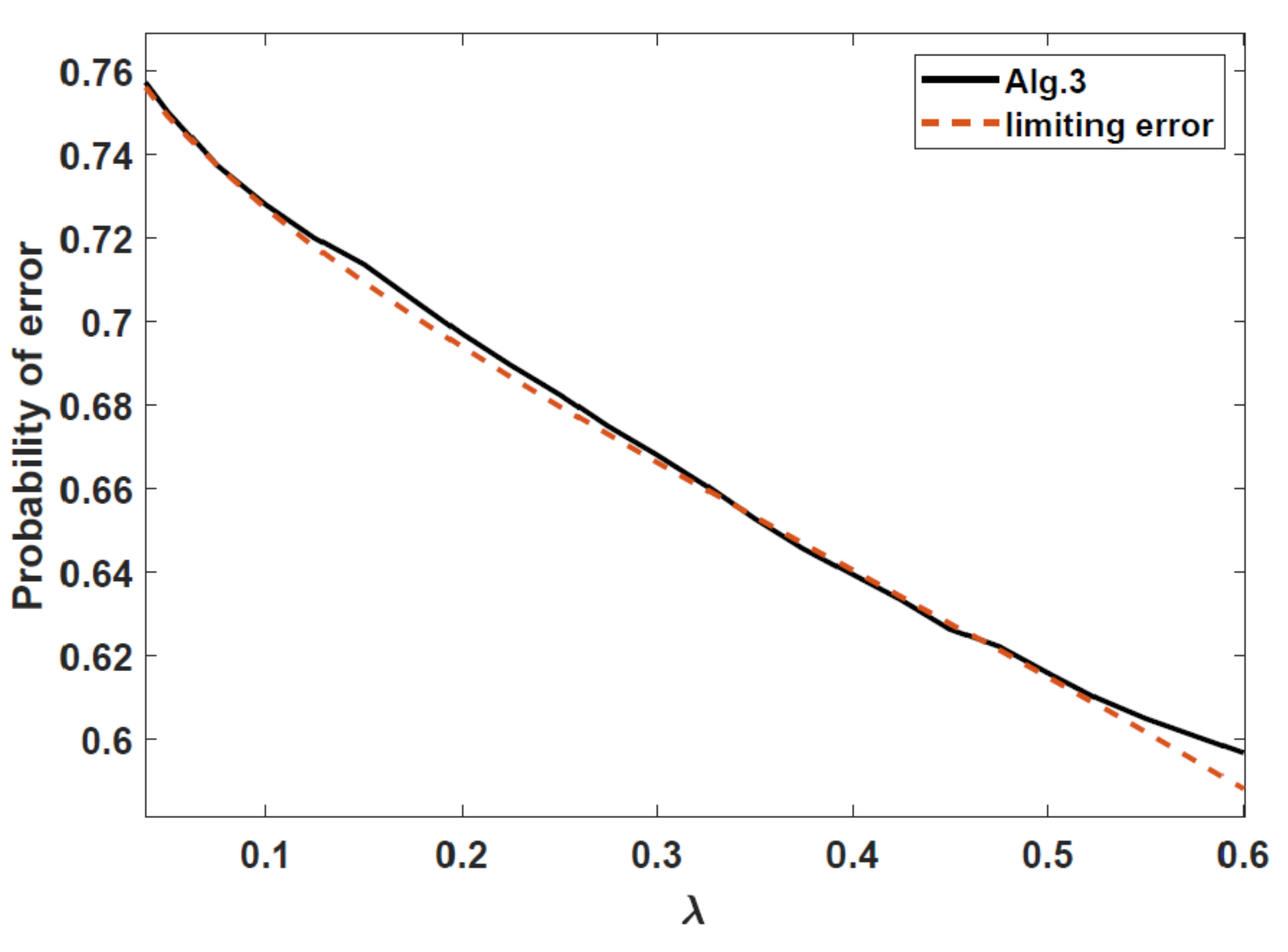}}
		\caption{The errors from the simulation with Algorithm \ref{alg:at} (solid) versus the limiting error \eqref{eq:limit_rank} (dashed) for the setting in Section \ref{subsec:rank}.}
		\label{fig:alg_at}
	\end{center}
	\vskip -0.2in
\end{figure}

\section{Analysis of the log likelihood ratio of spiked Gaussian Wigner matrices}\label{app:thm}

In Appendix \ref{app:thm}, we provide the detailed proof of Theorem \ref{thm:main}. 
We first briefly outline the proof for the case $k_1=0$, where the analysis is essentially the same as the rank-$1$ case in \cite{AlaouiJordan2018}. We then explain how we can generalize our proof to the nontrivial case $k_1\neq0$. 
We also remark that in case $k_1=0$ our analysis with appropriate changes can be applied to the case where the elements $x^{\ast}_i(\ell)$ of the given spike matrix have bounded supports and are independent with prior distributions $\caP_\ell$.

\begin{nrem}
	We use the standard big-O and little-o notation: $a_N = O(b_N)$ implies that there exists $N_0$ such that $a_N \le C b_N$ for some constant $C>0$ independent of $N$ for all $N \ge N_0$; $a_N = o(b_N)$ implies that for any positive constant $\varepsilon$ there exists $N_0$ such that $a_N \le \varepsilon b_N$ for all $N \ge N_0$.
	
	For $X$ and $Y$, which can be deterministic numbers and/or random variables depending on $N$, we use the notation $X = \caO(Y)$ if for any (small) $\varepsilon > 0$ and (large) $D > 0$ there exists $N_0 \equiv N_0 (\varepsilon, D)$ such that $\p(|X|>N^{\varepsilon} |Y|) < N^{-D}$ whenever $N > N_0$.
	
	For an event $\Omega$, we say that $\Omega$ holds with high probability if for any (large) $D > 0$ there exists $N_0 \equiv N_0 (D)$ such that $\p(\Omega^c) < N^{-D}$ whenever $N > N_0$.
	
	For a sequence of random variables, the notation $\Rightarrow$ denotes the convergence in distribution as $N\rightarrow\infty.$ 
\end{nrem}

\subsection{Proof of Theorem \ref{thm:main}}\label{app:thm1}

We first consider the case $k_1=0$ and $k_2=k$. The general case will be considered in Appendix \ref{app:extension}.  

Let $\caP_0 :=\caP^{\otimes k}$. 
By Bayes' rule, 
\[
\dd \bbP_2(X|Y)=\frac{e^{-H^k(X)}\dd\caP_0^{\otimes N}(X)}{\int e^{-H^k(X)}\dd\caP_0^{\otimes N}(X)}.
\]
For a positive integer $n$ and a function $f:(\bbR^{N\times k})^{n+1}\mapsto\bbR$, the Gibbs average of $f$ with respect to $H$ is defined as
\beq\label{gibbs_average_one}
\left<f(X^{(1)},\ldots,X^{(n)},X^*)\right>:=\frac{\int f(X^{(1)},\ldots,X^{(n)},X^*)\prod_{\ell=1}^{n}e^{-H^k(X^{(\ell)})}\dd\caP_0^{\otimes N}(X^{(\ell)})}{\left(\int e^{-H^k(X)}\dd\caP_0^{\otimes N}(X)\right)^{n}}.
\eeq
The variables $X^{(\ell)}, \ell=1\ldots,n$ are oftentimes called \emph{replicas}, which are random samples independently drawn from the posterior. Following \cite{AlaouiJordan2018}, we let
\beq
R_{\ell,\ell^\prime}(m,s):=\bsx^{(\ell)}(m)\cdot\bsx^{(\ell^\prime)}(s)=\frac1{N}\displaystyle\sum\limits_{i=1}^{N}x_i^{(\ell)}(m)x_i^{(\ell^\prime)}(s)
\eeq
for $\ell,\ell^\prime=1,\ldots,n,*$ and $m,s=1,2,\ldots,k$. 
The overlap for the rank-$k$ model is a $k\times k$ matrix
\beq\label{rank-k_overlap}
R^k_{1,*}=\frac{1}{N}X^{(1)T}X^{*}=\frac{1}{N}\sum_{i=1}^NX_i^{(1)T}X_i^{* }=[R_{1,*}(m,s)]_{1\le m,s\le k}.
\eeq

We remark that the Nishimori property (\cite{Nishimori2001}) holds for this model; the $(n+1)$-tuples $\left(X^{(1)},\ldots,X^{(n)},X^{(n+1)}\right)$ and $\left(X^{(1)},\ldots,X^{(n)},X^*\right)$ have the same distribution under $\E_{\bbP_2}\left<\cdot\right>$, which is the property in \cite{LelargeMiolane2016}. In particular, under $\bbP_2$, the distribution of the overlap $R^k_{1,*}$ between a replica and the spike is equal to that of the overlap $R^k_{1,2}$ between two replicas.

Recall that our proof of Theorem \ref{thm:main} is based on the fact that the characteristic function of the log-LR converges to a Gaussian. By differentiating the characteristic function $\phi_N$ defined in \eqref{char_log-LR}, we can readily prove Theorem \ref{thm:main} for the case $k_1=0$ by applying the following proposition that generalizes Proposition 9 of \cite{AlaouiJordan2018}.

\begin{prop}\label{prop:diagest}
	There exists a constant $\lambda_{0}(k,\caP)$ depending only on $k$ and the bound for the prior distribution $\caP$ such that for all $\lambda<\lambda_0$ and $s\in\bbR$,
	\[ \begin{split}
	&\E\left[\left(N\left<(R_{1,*}(\ell,m))^2\right>-\left<x_N(\ell)^2x_N^{*}(m)^2\right>\right)e^{is\log\caL}\right]=\frac{\lambda}{1-\lambda}\E\left[e^{is\log\caL}\right]+O(N^{-\frac{1}{2}}).
	\end{split} \]
\end{prop}

In the rest of Appendix \ref{app:thm}, we prove Proposition \ref{prop:diagest}.

\subsection{Preliminary bounds}\label{appe:PR}
As in \cite{AlaouiJordan2018}, we apply the interpolation trick for the proof of Proposition \ref{prop:diagest}. We collect a few results that will be repeatedly used in the proof. In what follows, we use the notation
\[
R^{-}_{\ell,\ell'}(m,s) = \frac{1}{N}\sum_{i=1}^{N-1} x_i^{(\ell)}(m)x_i^{(\ell')}(s), \qquad R^{k-}_{\ell,\ell'}=[R^{-}_{\ell,\ell'}(m,s)]_{1\le m,s\le k}.
\]
For a function $f$ of $n$ replicas $X^{(\ell)}$, $\ell=1,\cdots,n$, we use the notation of Talagrand,
\[
\nu_t(f) := \E\langle f\rangle_{t},
\]
where $\langle \cdot \rangle_t$ is the Gibbs average with respect to the family of (interpolating) Hamiltonians defined by
\[ \begin{split}
-H^k_t(X) 
&:= \sum_{1\le i < j \le N-1} \left( \sqrt{\frac{\lambda }{N}} W_{ij}X_iX_j^T + \frac{\lambda }{N}X_iX_j^TX_i^{*}X_j^{*T} -\frac{\lambda }{2N} (X_iX_j^T)^2 \right) \\
&\quad +\frac1{w_2} \sum_{i=1}^{N-1} \left( \sqrt{\frac{\lambda}{N}}W_{ii}X_iX_i^T+\frac{\lambda}{N}X_iX_i^TX_i^{* }X_i^{*T}-\frac{\lambda}{2N}(X_iX_i^T)^2 \right) \\
&\quad +\sum_{i=1}^{N-1} \left( \sqrt{\frac{\lambda t}{N}} W_{iN}X_i X_N^T + \frac{\lambda t}{N} X_iX_N^TX_i^{*} X_N^{*T} - \frac{\lambda t}{2N} (X_i X_N^T)^2 \right) \\
&\quad +\frac1{w_2}\sqrt{\frac{\lambda t}{N}}W_{NN}X_NX_N^T+\frac1{w_2}\frac{\lambda t}{N}X_NX_N^TX_N^{* }X_N^{*T}-\frac1{w_2}\frac{\lambda t}{2N}(X_NX_N^T)^2
\end{split} \]
for $t \in [0, 1]$. We set $\nu_1 \equiv \nu$. Note that $H_1^k = H$ and at $t=0$ the variable $x_N$ decouples from the other variables. 

We have the following formula for the derivative $\nu^\prime_t(f):= \dd \nu_t(f)/\dd t$. 
\begin{lem}\label{gibbs_derivative}
	Let $f$ be a function of $n$ replicas $X^{(1)},\cdots,X^{(n)}$ and $X^*$. Then
	$$\nu^\prime_t(f)=\sum_{m,s=1}^k\nu_t(f,m,s)$$
	where
	\beq \begin{split} \label{eq:nu_t}
		\nu_t(f,m,s) 
		&= \frac{\lambda }{2} \sum_{1\le \ell\neq \ell' \le n} \nu_t(R^{-}_{\ell,\ell'}(m,s) y^{(\ell)}(m)y^{(\ell')}(s)f) \\
		&\quad- \lambda  n \sum_{\ell=1}^n \nu_t(R^{-}_{\ell,n+1}(m,s) y^{(\ell)}(m)y^{(n+1)}(s)f) 
		\\&\quad + \lambda  \sum_{\ell=1}^n \nu_t(R^{-}_{\ell,*}(m,s) y^{(\ell)}(m)y^{*}(s)f)
		- \lambda  n  \nu_t(R^{-}_{n+1,*}(m,s) y^{(n+1)}(m)y^{*}(s)f) 
		\\&\quad +   \frac{\lambda n(n+1)}{2} \nu_t(R^{-}_{n+1,n+2}(m,s) y^{(n+1)}(m)y^{(n+2)}(s)f)
		\\&\quad +\frac{\lambda }{2w_2 N} \sum_{1\le \ell\neq \ell' \le n} \nu_t( y^{(\ell)}(m)^2y^{(\ell')}(s)^2f)
		- \frac{\lambda n}{2w_2 N} \sum_{\ell=1}^n \nu_t(y^{(\ell)}(m)^2y^{(n+1)}(s)^2f) 
		\\&\quad + \frac{\lambda}{w_2 N}  \sum_{\ell=1}^n \nu_t(y^{(\ell)}(m)^2y^{*}(s)^2f)
		- \frac{\lambda n}{w_2 N}  \nu_t(y^{(n+1)}(m)^2y^{*}(s)^2f) 
		\\&\quad + \frac{\lambda n(n+1)}{2w_2 N} \nu_t(y^{(n+1)}(m)^2y^{(n+2)}(s)^2f),
	\end{split} \eeq
	with $y := x_N$.
\end{lem}
\begin{proof}
	It follows from the Gaussian integration by parts. See, e.g., \cite{talagrand2011mean1}.
\end{proof}

\begin{rem}\label{rem:cont_overlap}
	We remark that the number of overlap terms in $\nu'_t$ is larger than that of $\nu_t$. Thus, with a sufficient decay rate of the moments of the overlaps, it is easier to control the error terms of the derivatives in Taylor's approximation than those of the original functions. 
\end{rem}

\begin{lem}\label{bound_gibbs_derivative}
	If $f$ is a bounded nonnegative function, then there exists a constant $K(\lambda, n)$ such that for all $t \in [0,1]$
	\[
	\nu_t(f) \le K(\lambda, n) \nu(f).
	\]
\end{lem}

\begin{proof}
	By Gr\"onwall's inequality, it suffices to show for that there exists a constant $K'(\lambda, n)$ such that for all $t\in [0,1]$
	\[
	|\nu_t'(f)| \le K'(\lambda, n) \nu_t(f),
	\]
	which follows from Lemma~\ref{gibbs_derivative} and that all variables and overlaps in \eqref{eq:nu_t} are bounded.
\end{proof}

\subsection{Proof of Proposition \ref{prop:diagest}\label{proof:prop}}


Following the proof of Proposition 9 in \cite{AlaouiJordan2018}, we consider self-consistent relations among various quantities. More precisely, we prove that for any $\lambda<1$,
\beq \label{eq:self1}
(1-\lambda)N\E\left[\left<(R_{1,*}(\ell,m))^2\right>e^{\ii s\log\caL}\right]=\E\left[\left<x_N(\ell)^2x_N^{*}(m)^2\right>e^{\ii s\log\caL}\right]+\delta_1,
\eeq
and
\beq \label{eq:self2}
\E\left[\left<x_N(\ell)^2x_N^{*}(m)^2\right>e^{\ii s\log\caL}\right]=\E\left[e^{\ii s\log\caL}\right]+\delta_2,
\eeq
where $|\delta_1|,|\delta_2|\le K(\lambda)N\max_{1\leq\ell,m\leq k}\E\langle |R_{1,*}(\ell,m)|^3\rangle$ for some constant $K(\lambda)$.
The main challenge of the proof is to show the following claim:

\emph{Claim.}
\beq \label{eq:third_claim}
\E \langle|R_{1,*}(m,s)|^3 \rangle \leq\frac{K(\lambda)}{N^{3/2}}
\eeq
for some constant $K(\lambda)$.

We remark that \eqref{eq:third_claim} is the optimal convergence rate of the third moment of each element of the overlap matrix $R^k_{1,*}$ under $\E \langle \cdot \rangle$. Once we have \eqref{eq:third_claim}, we can find $\delta_1, \delta_2 = O(N^{-1/2})$ and the desired result can be obtained as in Section 6 of \cite{AlaouiJordan2018}. For example, \eqref{eq:self1} can be proved by the cavity computation with the family of interpolating Hamiltonians $H^k_t$ for $t\in[0,1]$ and associated functions
\beq
X(t):=\exp\left(\ii s\log\displaystyle\int e^{-H^{k}_t(X)}\dd\caP_0^{\otimes N}(X)\right)
\eeq
and
\begin{align*}
&\varphi(t):=N\E\left[\left<x_N(\ell)x_N^*(m) R_{1,*}^{-}(\ell,m)\right>_tX(t)\right],
&&\psi(t)=	\E\left[\left<x_N(\ell)^2x_N^{*}(m)^2\right>_tX(t)\right]
\end{align*}
where $\left<\cdot\right>_t$ is the Gibbs average under the Hamiltonian $H^{k}_t.$
Note that, by symmetry among variables, 
\beq
\E\left[\left(N\left<(R_{1,*}(\ell,m))^2\right>-\left<x_N(\ell)^2x_N^{*}(m)^2\right>\right)e^{is\log\caL}\right]=\varphi(1).
\eeq
Then we obtain $\varphi(0)=0$, $\psi(0)=\E[X(0)]$ and 
\begin{align*}
\varphi^\prime(0) 
&=\lambda N\E\left[\left<x_N(\ell)^2x^*_N(m)^{2} (R_{1,*}^{-}(\ell,m))^2\right>_0X(0)\right] \\
&\quad +\frac{\lambda}{w_2}\E\left[\left<x_N(\ell)^3x^*_N(m)^{3} R_{1,*}^{-}(\ell,m)\right>_0X(0)\right]\\
&=\lambda N\E\left[\left< (R_{1,*}^{-}(\ell,m))^2\right>_0X(0)\right]+\frac{\lambda\kappa^2}{w_2}\E\left[\left< R_{1,*}^{-}(\ell,m)\right>_0X(0)\right]
\end{align*}
where $\kappa$ is the third moment of the given spike prior $\caP,$ since the $N$-th variable decouples from the rest of the Hamiltonian system at $t=0.$

We now apply Taylor's theorem to $\varphi(t)$ and $\psi(t)$ to find that
\begin{align*}
&|\varphi(1)-\varphi(0)-\varphi'(0)|=\caO(\delta),
&&|\psi(1)-\E[e^{\ii s\log\caL}]|=\caO(\delta).
\end{align*} 
Note that the we can prove the bounds
\[
|\varphi''(t)|,|\psi'(t)|,\E|X'(t)|=\caO(\delta).
\] 
for the second order terms, which easily follows from H\"{o}lder's inequality; see also Remark \ref{rem:cont_overlap}.

Similarly, we let
\begin{align}
&\psi_1(t)=\lambda N\E\left[\left< (R_{1,*}^{-}(\ell,m))^2\right>_tX(t)\right],
&&\psi_2(t)=\E\left[\left< R_{1,*}^{-}(\ell,m)\right>_tX(t)\right],
\end{align}
and find that
\begin{align*}
&|\psi_1(1)-\psi_1(0)|=\caO(\delta),
&&|\psi_2(t)|=\caO(N^{-1/2}).
\end{align*}
From the symmetry of variables,
\beq\begin{split}
	&\left|\psi_1(1)-\lambda N\E\left[\langle (R_{1,*}(\ell,m))^2\rangle e^{\ii s\log\caL}\right]\right| \\
	&\leq2\lambda\E\langle|x_N(\ell)x^{*}_N(m)R_{1,*}^{-}(\ell,m)|\rangle+\frac{\lambda}{N}\E\langle(x_N(\ell)x^{*}_N(m))^2\rangle\leq\frac{K(\lambda)}{\sqrt{N}},
\end{split}
\eeq
we then obtain \eqref{eq:self1}.

We now return to the proof of the claim \eqref{eq:third_claim}. While we need an in probability bounds of the overlaps as in Proposition 16 of \cite{AlaouiJordan2018} to prove the optimal rate of convergence of the overlaps up to the critical threshold, for our purpose, it suffices to prove the following result, which shows that the overlaps converge to zero at an optimal rate for sufficiently small SNR $\lambda$.

\begin{prop}\label{convergence_fourth_moment_symmetric_paramagnetic}
	For all $\lambda<\lambda_0(k,\caP)$ and $m,s=1,\ldots,k$, there exists a constant $K=K(\lambda) <\infty $ such that
	\[
	\E \langle (R_{1,*}(m,s))^{4}\rangle \le \frac{K}{N^{2}}.
	\]
\end{prop}

From Proposition \ref{convergence_fourth_moment_symmetric_paramagnetic} and H\"{o}lder's inequality we immediately obtain that
\[
\E\left\langle|R_{1,*}(m,s)|^3\right\rangle\le K(\lambda)\E\left\langle (R_{1,*}(m,s))^4\right\rangle^{3/4}\le\frac{K(\lambda)}{N^{3/2}},
\]
which proves \eqref{eq:third_claim}.

\subsection{Overlap convergence}\label{sec:overlap_0}

It remains to prove Proposition~\ref{convergence_fourth_moment_symmetric_paramagnetic}, which asserts the convergence of overlaps to zero under $\P_2$. We follow the algebraic cavity computation in \cite{Chen2018,AlaouiJordan2018}.
We begin by proving an estimate on the second moment,
\[
\E\langle (R_{1,*}(m,s))^2\rangle \le \frac{K}{N},
\]
which will result in the conclusion of Proposition~\ref{convergence_fourth_moment_symmetric_paramagnetic}. For the proof, we use the following two algebraic lemmas for the overlaps.

\begin{lem}\label{difference of overlap}
	For any $p\ge0$ and $m,s=1,2,\ldots,k$,
	\[
	\absv{(R_{1,*}(m,s))^{p+1}-(R^{-}_{1,*}(m,s))^{p+1}}\le\frac{C(p)}{N}\left(|R_{1,*}(m,s)|^{p}+|R^{-}_{1,*}(m,s)|^{p}\right)
	\]
\end{lem}

\begin{proof}
	The case of $p=0$ is obvious. It readily follows from an elementary inequality that 
	\[
	\absv{x^{p+1}-y^{p+1}}\le p\absv{x-y}(|x|^p+|y|^p)
	\]
	for any $x,y\in\bbR$ and $p\ge1$.
\end{proof}

\begin{lem}\label{detach from Nth}
	For a positive integer $p$, suppose that there exists a constant $C\ge1$ such that 
	\[
	\nu((R_{1,*}(m,s))^{2j})\le\frac{C}{N^j}
	\]
	for any $0\le j\le p$ and $m,s=1,2,\ldots,k.$
	Then,
	\[
	\nu((R^{-}_{1,*}(m,s))^{2p})\le\frac{C'(p)}{N^p}.
	\]
\end{lem}
\begin{proof}
	Since $R^{-}_{1,*}(m,s)=R_{1,*}(m,s)-\frac1{N}x_N(m)x^*_N(s)$, from the binomial expansion
	\[\begin{split}
	\nu((R^{-}_{1,*}(m,s))^{2p})
	&\le\sum_{j=0}^{2p}{{2p}\choose{j}}\frac1{N^{2p-j}}\nu(|R_{1,*}(m,s)|^{j}|x_N(m)x^*_N(s)|^{2p-j})\\
	&\le\sum_{j=0}^{2p}{{2p}\choose{j}}\frac{C(p)}{N^{2p-j}}\nu(|R_{1,*}(m,s)|^{j}). 	
	\end{split}
	\]
	Then, from the assumption of the lemma,
	\[ \begin{split}
	\nu((R^{-}_{1,*}(m,s))^{2p})\le C\sum_{j=0}^{2p}{{2p}\choose{j}}\left(\frac{C(p)}{N}\right)^{2p-j}\left(\frac{1}{\sqrt{N}}\right)\le C\left(\frac{C(p)}{N}+\frac1{\sqrt{N}}\right)^{2p}\le \frac{C'(p)}{N^p},
	\end{split}
	\]
	where we used the Schwarz inequality $\nu(|R_{1,*}(m,s)|^{j})\le\nu(|R_{1,*}(m,s)|^{2j})^{\frac1{2}}$.
\end{proof}

\begin{lem}\label{main_bound_for_second}
	There exist constants $C_1$ and $C_2$, independent of $N$ and $\lambda$, such that
	\beq
	\begin{split}
		\max_{1\le m,s\le k}\nu((R_{1,*}(m,s))^{2})&\le C_1 \lambda^2 \max_{1\le m,s\le k}\nu(|R_{1,*}(m,s)|^{3})\\
		&\quad +\lambda \max_{1\le m,s\le k}\nu((R_{1,*}(m,s))^{2})+\frac{C_2 \lambda}{N}.
	\end{split}
	\eeq
\end{lem} 
\begin{proof}
	By the symmetry between the variables, we have
	\beq
	\nu( (R_{1,*}(m,s))^{2}) = \nu( x_N(m) x_N^*(s) (R_{1,*}(m,s)))=\nu( x_N(m) x_N^*(s) (R^{-}_{1,*}(m,s)))+ \Delta
	\eeq
	where $\Delta=\nu(x_N(m) x_N^*(s)((R_{1,*}(m,s))-(R^{-}_{1,*}(m,s))))$.
	Then, by Lemmas~\ref{difference of overlap} and \ref{detach from Nth},
	\beq
	|\Delta|\le C\nu\left(\absv{(R_{1,*}(m,s))-(R^{-}_{1,*}(m,s))}\right)\le\frac{C}{N},
	\eeq
	and thus
	\beq \label{eq:B.17}
	\nu( (R_{1,*}(m,s))^2)\le\nu( x_N(m) x_N^*(s) (R^{-}_{1,*}(m,s)))+\frac{C'(p)}{N}
	\eeq
	for some constants $C$ and $C' =C'(p)$. (In this proof, $C$ and $C'$ will denote various constants independent of $N$ that may differ line by line.)
	
	To estimate the first term of the right side of \eqref{eq:B.17}, we let $f=x_N(m) x_N^*(s) (R^{-}_{1,*}(m,s))$ and apply Lemma~\ref{gibbs_derivative}. Note that $\nu_0(f)=0,$ since $\caP$ is centered. By Taylor's theorem,
	\beq\label{eq:Taylor}
	\nu( (R_{1,*}(m,s))^2)\le\nu(f)+\frac{C}{N}\le \nu'_0(f)+\frac1{2}\sup_{0\le t\le1}|\nu''_t(f)|+\frac{C}{N}.
	\eeq
	We now estimate $\nu'_0(f)$. From Lemma~\ref{gibbs_derivative} with 1-replica, we have that
	\begin{align}
	\nu_0'(f) 
	&= -\sum_{m',s'=1}^k\lambda \nu_0(A(1,2,m',s',m,s))+\sum_{m',s'=1}^k\lambda \nu_0(A(1,*,m',s',m,s))\nonumber\\
	&~~~ - \sum_{m',s'=1}^k\lambda \nu_0(A(2,*,m',s',m,s))+  \sum_{m',s'=1}^k \lambda \nu_0(A(2,3,m',s',m,s))\nonumber\\
	&~~~- \sum_{m',s'=1}^k \frac{\lambda}{2w_2 N} \nu_0(B(1,2,m',s',m,s))+\sum_{m',s'=1}^k \frac{\lambda}{w_2 N} \nu_0(B(1,*,m',s',m,s))\nonumber
	\\&~~~-\sum_{m',s'=1}^k \frac{\lambda }{w_2 N}  \nu_0(B(2,*,m',s',m,s)) +\sum_{m',s'=1}^k \frac{\lambda}{w_2 N} \nu_0(B(2,3,m',s',m,s))\nonumber
	\\&= \lambda \nu_0((R^{-}_{1,*}(m,s))^2)+ \frac{\lambda\kappa^2}{w_2 N} \nu_0((R^{-}_{1,*}(m,s))),\label{eq:remaining term}
	\end{align}
	where we use the notation
	\[
	A(a,b,m',s',m,s)=y^{(a)}(m')y^{(b)}(s')y^{(1)}(m)y^*(s)(R^-_{a,b}(m',s'))(R^-_{1,*}(m,s))
	\] 
	and 
	\[
	B(a,b,m',s',m,s)=y^{(a)}(m')^2y^{(b)}(s')^2y^{(1)}(m)y^*(s)(R^-_{1,*}(m,s)).
	\]
	Note that the second term in the right-side of \eqref{eq:remaining term} is $O(N^{-1})$ since $R^-_{1,*}(m,s)$ is bounded.
	
	We now turn to $\nu''_t(f)$. Using Lemma~\ref{gibbs_derivative} recursively, we find that $\nu''_t(f)$ is represented by a linear combination of functions of the following forms:
	\begin{itemize}
		\item $\lambda^2\nu_t(R^-_{\ell_1,\ell_2}(m_1,s_1)R^-_{\ell_3,\ell_4}(m_2,s_2)y^{(\ell_1)}(m_1)y^{(\ell_2)}(s_1)y^{(\ell_3)}(m_2)y^{(\ell_4)}(s_2)f)$
		\item $\frac{\lambda^2}{N}\nu_t(R^-_{\ell_1,\ell_2}(m_1,s_1)y^{(\ell_1)}(m_1)y^{(\ell_2)}(s_1)y^{(\ell_3)}(m_2)^2y^{(\ell_4)}(s_2)^2f)$
		\item $\frac{\lambda^2}{N^2}\nu_t(y^{(\ell_1)}(m_1)^2y^{(\ell_2)}(s_1)^2y^{(\ell_3)}(m_2)^2y^{(\ell_4)}(s_2)^2f)$
	\end{itemize}
	where $\ell_1\neq\ell_2$ and $\ell_3\neq\ell_4$.
	For the terms of the first form, we have that for any $1\le\ell_1\neq\ell_2\le n$, $1\le\ell_3\neq\ell_4\le n$ and $1\le m_1,s_1,m_2,s_2\le k$
	\begin{align*}
	\nu_t(|R^{-}_{1,*}(m,s)|
	&|R^-_{\ell_1,\ell_2}(m_1,s_1)||R^-_{\ell_3,\ell_4}(m_2,s_2)|)\\
	&\le\nu_t(|R^{-}_{1,*}(m,s)|^3)^{1/3}\nu_t(|R^-_{\ell_1,\ell_2}(m_1,s_1)|^3)^{1/3}\nu_t(|R^-_{\ell_3,\ell_4}(m_2,s_2)|^3)^{1/3}\\
	&\le C\max_{1\le m,s\le k}\nu(|R_{1,*}^-(m,s)|^{3}),
	\end{align*}
	where we used the generalized H\"{o}lder's inequality, the Nishimori property and Lemma~\ref{bound_gibbs_derivative}. The other terms are obviously $O(N^{-1})$.
	Further, by Lemma~\ref{difference of overlap}, for any $h\in\N$
	\begin{align*}
	\nu(|R_{1,*}^-(m,s)|^{h+1})
	&\le\nu(|R_{1,*}(m,s)|^{h+1})+\frac{C}{N}(\nu(|R_{1,*}(m,s)|^h)+\nu(|R_{1,*}^-(m,s)|^h))
	\\&\le\nu(|R_{1,*}(m,s)|^{h+1})+\frac{C}{N}.
	\end{align*}
	Together with \eqref{eq:Taylor}, \eqref{eq:remaining term}, we thus have
	\beq \begin{split} \label{eq:B18}
		\nu( (R_{1,*}(m,s))^2)
		\le \lambda^2c_1\max_{1\le m,s\le k} \nu(|R_{1,*}(m,s)|^{3}) +\lambda \max_{1\le m\le k}\nu_0((R_{1,*}^-(m,s))^{2})+\frac{\lambda c_2}{N}
	\end{split} \eeq
	for some constants $c_1,c_2$. 
	
	We now control the second term in the right side of \eqref{eq:B18}. 
	Applying the same argument to $f=(R_{1,*}^-(m,s))^2$,
	\[
	\nu_0((R_{1,*}^-(m,s))^{2})\le \nu((R_{1,*}^-(m,s))^{2})+\sup_{0\le t\le1}|\nu'_t((R_{1,*}^-(m,s))^{2})|.
	\]
	Notice that
	\[
	|\nu'_t((R_{1,*}^-(m,s))^2)|\le C\lambda \max_{1\le m,s\le k}\nu(|R_{1,*}(m,s)|^{3})+\frac{C'}{N}
	\]
	and
	\[
	\nu((R_{1,*}^-(m,s))^{2})\le\nu((R_{1,*}(m,s))^{2})+\frac{C}{N}.
	\]
	Thus, from \eqref{eq:B18},
	\[\begin{split}
	\nu( (R_{1,*}(m,s))^2)
	\le C_1 \lambda^2 \max_{1\le m,s\le k}\nu(|R_{1,*}(m,s)|^{3}) +\lambda \max_{1\le m\le k}\nu((R_{1,*}(m,s))^{2})+\frac{C_2 \lambda }{N}
	\end{split}\]
	for some constants $C_1$ and $C_2$, independent of $N$ and $\lambda$.
	This concludes the proof of the desired lemma.
\end{proof}

From Lemma \ref{main_bound_for_second}, we obtain the following lemma.
\begin{lem}\label{lem:second moment} 
	There exists a positive constant $\lambda_{0,1}$ such that for any $\lambda<\lambda_{0,1}$
	\beq\label{limsup_square}
	\limsup_{N\to\infty} \; N\max_{1\le m,s\le k}\nu((R_{1,*}(m,s))^2)<\infty
	\eeq
\end{lem}

\begin{proof} 
	Since the overlap is trivially bounded by some fixed constant, from Lemma~\ref{main_bound_for_second}, we find that
	\begin{align*}
	\max_{1\le m,s\le k}\nu((R_{1,*}(m,s))^{2})
	\le (\wt C_1\lambda +1)\lambda  \max_{1\le m,s\le k}\nu((R_{1,*}(m,s))^{2})+\frac{C_2\lambda}{N},
	\end{align*}
	where $\wt C_1$ is the constant which depends on $C_1$ in Lemma \ref{main_bound_for_second} and the bound of $\caP$. Thus, if
	\[
	\lambda < \lambda_{0,1}:= \frac{2}{\sqrt{1+4\wt C_1} +1},
	\]
	then
	\beq
	\max_{1\le m,s\le k}\nu((R_{1,*}(m,s))^{2})\le\frac{C_2 \lambda}{N(1- (\wt C_1\lambda+1)\lambda)}.
	\eeq
	This proves the desired lemma.
\end{proof}

With the bound from Lemma \ref{lem:second moment}, the corresponding result for the fourth moment can be proved in a similar manner, and we only state the series of lemmas that lead us to the conclusion.

\begin{lem}\label{main_bound_for_fourth}
	There exists constants $K_1$ and $K_2$, independent of $N$ and $\lambda$, such that
	\beq
	\begin{split}
		&\max_{1\le m,s\le k}\nu((R_{1,*}(m,s))^{4})\\
		&\le K_1\lambda^2\max_{1\le m,s\le k}\nu(|R_{1,*}(m,s)|^{5})+\lambda \max_{1\le m,s\le k}\nu((R_{1,*}(m,s))^{4})+\frac{K_2\lambda}{N^2}.
	\end{split}
	\eeq  
\end{lem}

\begin{lem} \label{lem:fourth_moment}
	There exists a positive constant $\lambda_{0,2} \equiv \lambda_0$ such that for any $\lambda<\lambda_{0,2}$,
	\beq
	\limsup_{N\to\infty} \; N^2\max_{1\le m,s\le k}\nu((R_{1,*}(m,s))^4)<\infty.
	\eeq
\end{lem}

\subsection{Extension to the case $k_1\neq 0$}\label{app:extension}
In this section, we prove Theorem \ref{thm:main} for the case $k_1\neq 0$. Recall that 
\[
\phi_N(\lambda)=\E_{\bbP_2}e^{\ii s\log\caL(Y;k_1,k_2)}=\E_{\bbP_2}e^{\ii s\log\caL(Y;k_2)-\log\caL(Y;k_1)}:=\E e^{\ii s\log\caL}.
\]
As in the case of $k_1=0$, we assume $\bsH_2$. 
For simplicity, we first consider the case where there are no diagonal elements as in \cite{AlaouiJordan2018}. 

In what follows, we present some notations for convenience:
\begin{itemize}
	\item We denote by $\langle\cdot\rangle^{[1]}$, $\langle\cdot\rangle^{[2]}$, and $\langle\cdot\rangle^{[1,2]}$ the Gibbs measures with respect to the Hamiltonians $-H^{k_1}(X^{[1]})$, $-H^{k_2}(X^{[2]})$ and $-H^{k_1,k_2}(X^{[1]},X^{[2]})=-H^{k_1}(X^{[1]})-H^{k_2}(X^{[2]})$, respectively. 
	\item Correspondingly, we use Talagrand's notations $\nu^{[1]}_t(f)=\E\langle\cdot\rangle^{[1]}_{t}$, $\nu^{[2]}_t(f)=\E\langle\cdot\rangle^{[2]}_{t}$ and $\nu^{[1,2]}_t(f)=\E\langle\cdot\rangle^{[1,2]}_{t}$ for the Gibbs measures $\langle\cdot\rangle^{[1]}_t$, $\langle\cdot\rangle^{[2]}_t$ and $\langle\cdot\rangle^{[1,2]}_t$ with respect to the interpolating Hamiltonians $-H_t^{k_1}(X^{[1]})$, $-H_t^{k_2}(X^{[2]})$ and $-H_t^{k_1,k_2}(X^{[1]},X^{[2]})=-H_t^{k_1}(X^{[1]})-H_t^{k_2}(X^{[2]})$, respectively. 
	\item We denote the overlaps by $R^{[1]}_{\ell,\ell'}(m,n)=\frac{1}{N}\sum_{i=1}^{N} x^{[1],(\ell)}_i(m)x_i^{[1],(\ell')}(n)$, $R^{[2]}_{\ell,\ell'}(m,n)=\frac{1}{N}\sum_{i=1}^{N} x_i^{[2],(\ell)}(m)x_i^{[2],(\ell')}(n)$ and $R^{[1,2]}_{\ell,\ell'}(m,n)=\frac{1}{N}\sum_{i=1}^{N} x_i^{[1],(\ell)}(m)x_i^{[2],(\ell')}(n)$.
\end{itemize}
Applying Stein's lemma, we have 
\beq\label{eq:derivative_char}
\begin{split}
	\phi'_N(\lambda)
	&=\frac{s^2}{2}\sum_{m=1}^{k_1}\sum_{n=1}^{k_2}\E\left[(N\langle R^{[1,2]}_{1,1}(m,n)^2\rangle^{[1,2]}-\langle x^{[1]}_N(m)^2x^{[2]}_N(n)^2\rangle^{[1,2]})e^{\ii s\log\caL}\right]\\
	&~~~-\frac{ s^2}{4}\sum_{m=1}^{k_2}\sum_{n=1}^{k_2}\E\left[(N\langle R^{[2]}_{1,2}(m,n)^2\rangle^{[2]}-\langle x^{[2],(1)}_N(m)^2x^{[2],(2)}_N(n)^2\rangle^{[2]})e^{\ii s\log\caL}\right]\\
	&~~~-\frac{s^2}{4}\sum_{m=1}^{k_1}\sum_{n=1}^{k_1}\E\left[(N\langle R^{[1]}_{1,2}(m,n)^2\rangle^{[1]}-\langle x^{[1],(1)}_N(m)^2x^{[1],(2)}_N(n)^2\rangle^{[1]})e^{\ii s\log\caL}\right]\\
	&~~~-\frac{\ii s}{4}\sum_{m=1}^{k_2}\sum_{n=1}^{k_2}\E\left[(N\langle R^{[2]}_{1,2}(m,n)^2\rangle^{[2]}-\langle x^{[2],(1)}_N(m)^2x^{[2],(2)}_N(n)^2\rangle^{[2]})e^{\ii s\log\caL}\right]\\
	&~~~+\frac{\ii s}{4}\sum_{m=1}^{k_1}\sum_{n=1}^{k_1}\E\left[(N\langle R^{[1]}_{1,2}(m,n)^2\rangle^{[1]}-\langle x^{[1],(1)}_N(m)^2x^{[1],(2)}_N(n)^2\rangle^{[1]})e^{\ii s\log\caL}\right]\\
	&~~~+\frac{\ii s}{2}\sum_{m=1}^{k_2}\sum_{n=1}^{k_2}\E\left[(N\langle R^{[2]}_{1,*}(m,n)^2\rangle^{[2]}-\langle x^{[2]}_N(m)^2x^{*}_N(n)^2\rangle^{[2]})e^{\ii s\log\caL}\right]\\
	&~~~-\frac{\ii s}{2}\sum_{m=1}^{k_1}\sum_{n=1}^{k_2}\E\left[(N\langle R^{[1]}_{1,*}(m,n)^2\rangle^{[1]}-\langle x^{[1]}_N(m)^2x^{*}_N(n)^2\rangle^{[1]})e^{\ii s\log\caL}\right].
\end{split}
\eeq
Let 
\[ \begin{split}
A(p)= \max\Big\{ &\max\limits_{m,n}\nu^{[1]}(|R^{[1]}_{1,2}(m,n)|^p), \max\limits_{m,n}\nu^{[1]}(|R^{[1]}_{1,\ast}(m,n)|^p), \\
&\max\limits_{m,n}\nu^{[2]}(|R^{[2]}_{1,2}(m,n)|^p), \max\limits_{m,n}\nu^{[1,2]}(|R^{[1,2]}_{1,1}(m,n)|^p) \Big\}
\end{split} 
\]
be the maximum of the $p$-th moments of all possible overlaps. Note that 
\[
\nu^{[2]}(|R^{[2]}_{1,2}(m,n)|^p)=\nu^{[2]}(|R^{[2]}_{1,\ast}(m,n)|^p)
\]
by the Nishimori property under $\E_{\bbP_2}\langle\cdot\rangle^{[2]}$, and 
\[
\nu^{[1]}(|R^{[1]}_{1,2}(m,n)|^p)=\nu^{[1,2]}(|R^{[1]}_{1,2}(m,n)|^p).
\]

Our first goal is to prove that the maximal moments $A(2)$ and $A(4)$ vanish at an optimal rate when $\lambda$ is sufficiently small. We note that all algebraic lemmas in Section \ref{sec:overlap_0} also hold for all overlaps shown in the $k_1\neq0$ case. For the proof, we need the following lemma that generalizes 
Lemma \ref{gibbs_derivative}.
\begin{lem}\label{gibbs_derivative2}
	Let $f$ be a function of $n$ replicas $X^{[1],(1)},\cdots,X^{[1],(n)}$, $X^{[2],(1)},\cdots,X^{[2],(n)}$ and $X^*$. Then
	\begin{align}
	\nu^{[1,2]\prime}_t(f) 
	&= \frac{\lambda }{2} \sum_{i=1}^2\sum_{m_i,s_i=1}^{k_i}\sum_{1\le \ell\neq \ell' \le n} \nu_t^{[1,2]}(R^{[i]-}_{\ell,\ell'}(m_i,s_i) y^{[i],(\ell)}(m_i)y^{[i],(\ell')}(s_i)f) 
	\nonumber
	\\&+\lambda \sum_{m=1}^{k_1}\sum_{s=1}^{k_2}\sum_{\ell, \ell'=1}^n \nu_t^{[1,2]}(R^{[1,2]-}_{\ell,\ell'}(m,s) y^{[1],(\ell)}(m)y^{[2],(\ell')}(s)f)\label{derivative2}
	\\& - \lambda  n \sum_{i=1}^2\sum_{m_i,s_i=1}^{k_i}\sum_{\ell=1}^n \nu_t^{[1,2]}(R^{[i]-}_{\ell,n+1}(m_i,s_i) y^{[i],(\ell)}(m_i)y^{[i],(n+1)}(s_j)f) 
	\nonumber
	\\& - \frac{\lambda n}{2}   \sum_{m=1}^{k_1}\sum_{s=1}^{k_2}\sum_{\ell=1}^n \nu_t^{[1,2]}(R^{[1,2]-}_{n+1,\ell}(m,s)y^{[1],(n+1)}(m) y^{[2],(\ell)}(s)f) 
	\nonumber
	\\& - \frac{\lambda n}{2}   \sum_{m=1}^{k_1}\sum_{s=1}^{k_2}\sum_{\ell=1}^n \nu_t^{[1,2]}(R^{[1,2]-}_{\ell,n+1}(m,s)y^{[1],(\ell)}(m) y^{[2],(n+1)}(s)f) 
	\nonumber
	\\&+ \lambda  \sum_{i=1}^2\sum_{m_i=1}^{k_i}\sum_{s=1}^{k_2}\sum_{\ell=1}^n \nu_t^{[1,2]}(R^{[i]-}_{\ell,*}(m_i,s) y^{[i],(\ell)}(m_i)y^{*}(s)f)
	\nonumber
	\\&- \lambda  n\sum_{i=1}^2 \sum_{m_i=1}^{k_i}\sum_{s=1}^{k_2} \nu_t^{[1,2]}(R^{[i]-}_{n+1,*}(m_i,s) y^{(n+1)}(m_i)y^{*}(s)f) 
	\nonumber
	\\&+  \frac{\lambda n(n+1)}{2} \sum_{i=1}^2\sum_{m_i,s_i=1}^{k_i}\nu_t^{[1,2]}(R^{[i]-}_{n+1,n+2}(m_i,s_i) y^{(n+1)}(m_i)y^{(n+2)}(s_j)f)
	\nonumber
	\\&+  \frac{\lambda n(n+1)}{2} \sum_{m=1}^{k_1}\sum_{s=1}^{k_2}\nu_t^{[1,2]}(R^{[1,2]-}_{n+1,n+2}(m,s) y^{[1],(n+1)}(m)y^{[2],(n+2)}(s)f)
	\nonumber
	\\&+  \frac{\lambda n(n+1)}{2} \sum_{m=1}^{k_1}\sum_{s=1}^{k_2}\nu_t^{[1,2]}(R^{[1,2]-}_{n+2,n+1}(m,s) y^{[1],(n+2)}(m)y^{[2],(n+1)}(s)f)
	\nonumber
	\end{align}
	with $y = x_N$.
\end{lem}
Using Lemma \ref{gibbs_derivative2} instead of Lemma \ref{gibbs_derivative}, we have the following analogues of Lemmas \ref{main_bound_for_second} and \ref{main_bound_for_fourth}.

\begin{lem}\label{lem:inequalities} 
	There exist nonnegative constants $C_1$, $C_2$, $K_1$, and $K_2$, independent of $N$ and $\lambda$, such that
	\beq\label{second_moment_k_1>0}
	A(2)\leq C_1 \lambda^2 A(3)+\lambda A(2)+\frac{C_2 \lambda}{N}
	\eeq
	and 
	\beq\label{fourth_moment_k_1>0}
	A(4)\leq K_1 \lambda^2 A(5)+\lambda A(4)+\frac{K_2 \lambda}{N^2}.
	\eeq
\end{lem}
\begin{proof} 
	We set 
	\beq
	f_{1,s}=x^{[1],(1)}_N(m)x^{[1],(2)}_N(n)(R^{[1]-}_{1,2}(m,n))^s,
	\eeq
	\beq 
	f_{2,s}=x^{[1]}_N(m)x^{\ast}_N(n)(R^{[1]-}_{1,\ast}(m,n))^s
	\eeq
	
	\beq
	f_{3,s}=x^{[2],(1)}_N(m)x^{[2],(2)}_N(n)(R^{[2]-}_{1,2}(m,n))^s
	\eeq 
	and 
	\beq
	f_{4,s}=x^{[1]}_N(m)x^{[2]}_N(n)(R^{[1,2]-}_{1,1}(m,n))^s.
	\eeq 
	
	First, we consider $f_{4,1}=x^{[1],(1)}_N(m)x^{[2],(1)}_N(n)(R^{[1,2]-}_{1,1}(m,n))$. By Taylor's theorem and the symmetry of variables, 
	\begin{align}\label{first_k_1>0}
	\nu^{[1,2]}((R^{[1,2]}_{1,1}(m,n))^2)\leq\nu_0^{[1,2]\prime}(f_{4,1})+\frac{C}{N}+\frac1{2}\sup_{0\leq t\leq1}|\nu_t^{[1,2]\prime\prime}(f_{4,1})|.
	\end{align}
	As in Lemma \ref{main_bound_for_second}, using Lemma \ref{gibbs_derivative2} in place of Lemma \ref{gibbs_derivative} with 1-replica, the remaining term of $\nu^{[1,2]\prime}_0(f_{4,1})$ is from \eqref{derivative2} and it is 
	\begin{align}\label{second_k_1>0}
	\nu_0^{[1,2]\prime}(f_{4,1})=\lambda\nu_0^{[1,2]}((R^{[1,2]-}_{1,1}(m,n))^2).
	\end{align}
	Following the computation as in the case of $k_1=0$, we observe that the second derivative is a linear combination of 
	\begin{equation*}
	\lambda^2\nu^{[1,2]}_t(x^{[p_1],(q_1)}_N(a)x^{[p_2],(q_2)}_N(b)x^{[r_1],(s_1)}_N(c)x^{[r_2],(s_2)}_N(d)(R^{[p_1,q_1]-}_{p_2,q_2}(a,b))(R^{[r_1,s_1]-}_{r_2,s_2}(c,d))f_{4,1})
	\end{equation*}
	for admissible indices.
	From the representation above of the second derivative, we obtain   
	\begin{align}\label{third_k_1>0}
	|\nu_t^{[1,2]\prime\prime}(f_{4,1})|\leq \lambda^2c_1A(3)+\frac{\lambda c_2}{N}
	\end{align}
	by Lemmas \ref{gibbs_derivative2}, \ref{difference of overlap}, \ref{bound_gibbs_derivative} and H\"{o}lder's inequality. Using the first order Taylor approximation for $\nu_t^{[1,2]}((R^{[1,2]-}_{1,1}(m,n))^2)$, we also see that 
	\[
	\lambda\nu_0^{[1,2]}((R^{[1,2]-}_{1,1}(m,n))^2)\le \lambda\nu^{[1,2]}((R^{[1,2]}_{1,1}(m,n))^2)+\lambda^2 c'_1A(3)+\frac{\lambda c'_2}{N}
	\]
	since, as in the $k_1=0$ case, its first derivative also bounded by 
	\[
	|\nu'_t((R^{[1,2]-}_{1,1}(m,n))^2)|\le \lambda c_1''A(3)+\frac{\lambda c''_2}{N}
	\]
	Putting the above results together, we have 
	\begin{align}
	\max_{m,n}\nu^{[1,2]}((R^{[1,2]}_{1,1}(m,n))^2)\leq  C_1 \lambda^2 A(3)+\lambda\max_{m,n}\nu^{[1,2]}((R^{[1,2]}_{1,1}(m,n))^2)+\frac{C_2 \lambda}{N}.
	\end{align}
	From the definition of the Gibbs average, we directly see that
	\beq
	\nu^{[1]}(|R^{[1]}_{1,2}(m,n)|^p)=\nu^{[1,2]}(|R^{[1]}_{1,2}(m,n)|^p),
	\eeq
	\beq
	\nu^{[1]}(|R^{[1]}_{1,\ast}(m,n)|^p)=\nu^{[1,2]}(|R^{[1]}_{1,\ast}(m,n)|^p)
	\eeq
	and
	\beq
	\nu^{[2]}(|R^{[2]}_{1,2}(m,n)|^p)=\nu^{[1,2]}(|R^{[2]}_{1,2}(m,n)|^p).
	\eeq 
	Repeating the exactly same procedure for the other overlaps, we obtain the desired inequality \eqref{second_moment_k_1>0}. 
	
	It remains to prove \eqref{fourth_moment_k_1>0}. Using the functions $\{f_{i,3}\}_{1\leq i\leq4}$ instead of $\{f_{i,1}\}_{1\leq i\leq4}$, it is easy to obtain the desired inequality. We omit the details.
\end{proof}
Thus, applying the same argument as in Lemmas \ref{lem:second moment} and \ref{lem:fourth_moment}, we also obtain that for any $\lambda<\lambda_0$
\beq \label{eq:moments_bound}
A(2)=\caO(N^{-1}), \qquad A(4)=\caO(N^{-2})
\eeq
for some $\lambda_0>0$ that depends on $k_1$, $k_2$ and the bound for the prior distribution $\caP$ but not on $N$.

We now estimate the terms in \eqref{eq:derivative_char}. 
\begin{prop}\label{prop:est}
	Let $f$ be all terms that appeared in \eqref{eq:derivative_char}.
	Then, there exists a constant $\lambda_0$ such that for all $\lambda<\lambda_0$ and $s\in\R,$
	\beq
	\E\left[\langle f\rangle^{[1,2]}e^{\ii s\log\caL}\right]=\frac{\lambda}{1-\lambda}\E[e^{\ii s\log\caL}]+\caO(N^{-1/2}).
	\eeq
\end{prop}
\begin{proof}
	In this proof, we only consider the case $f=NR^{[1,2]}_{1,1}(m,n)^2-x^{[1]}_N(m)^2x^{[2]}_N(n)^2$; the corresponding results for the other terms can be proved in a similar manner.
	
	From the symmetry of variables 
	\[\begin{split}
	&\E\left[(N\langle (R^{[1,2]}_{1,1}(m,n))^2\rangle^{[1,2]}-\langle x^{[1]}_N(m)^2x^{[2]}_N(n)^2\rangle^{[1,2]})e^{\ii s\log\caL}\right]
	\\&=N\E\left[\langle x^{[1]}_N(m)x^{[2]}_N(n)R^{[1,2]-}_{1,1}(m,n)\rangle^{[1,2]}e^{\ii s\log\caL}\right].
	\end{split}
	\]
	Let 
	\beq
	X(t)=\exp\left(\ii s\log\int e^{-H_{t}^{k_2}(X^{[2]})}\dd\caP_{0,2}^{\otimes N}(X^{[2]})-\ii s\log\int e^{-H_{t}^{k_1}(X^{[1]})}\dd\caP_{0,1}^{\otimes N}(X^{[1]})\right)
	\eeq
	and 
	\begin{align*}
	&\varphi(t)=N\E\left[\langle x^{[1]}_N(m)x^{[2]}_N(n)R^{[1,2]-}_{1,1}(m,n)\rangle^{[1,2]}_{t}X(t)\right],
	&&\psi(t)=\E\left[\langle x^{[1]}_N(m)^2x^{[2]}_N(n)^2\rangle^{[1,2]}_{t}X(t)\right]
	\end{align*}
	From the bounds \eqref{eq:moments_bound}, we find that the second derivative of $\varphi$ is bounded by 
	\beq
	\sup_{0\le t\le1}|\varphi''(t)|=\caO(N^{-1/2}).
	\eeq
	Since $\varphi(0)=0$, from Taylor's theorem, $\varphi(1)= \varphi'(0) + \caO(N^{-1/2})$ .
	Further, we can check that 
	\beq\begin{split}
		\varphi'(0)
		&=\lambda N\E\left[\langle x^{[1]}_N(m)^2x^{[2]}_N(n)^2(R^{[1,2]-}_{1,1}(m,n))^2\rangle^{[1,2]}_{0}X(0)\right]
		\\&=\lambda N\E\left[\langle (R^{[1,2]-}_{1,1}(m,n))^2\rangle^{[1,2]}_{0}X(0)\right].
	\end{split}
	\eeq
	Now, we consider the function 
	\begin{align*}
	\psi_1(t):=\lambda N\E\left[\langle (R^{[1,2]-}_{1,1}(m,n))^2\rangle^{[1,2]}_{t}X(t)\right].
	\end{align*} 
	Similarly, we also find that $\psi_1(0)=\varphi'(0)$ and $\sup_{0\le t\le 1}|\psi_1'(t)|=\caO(N^{-1/2})$.
	From the symmetry of variables,  
	\beq\begin{split}
		&\left|\psi_1(1)-\lambda N\E\left[\langle (R^{[1,2]}_{1,1}(m,n))^2\rangle^{[1,2]}e^{\ii s\log\caL}\right]\right| \\
		&\leq2\lambda\E\langle|x^{[1]}_N(m)x^{[2]}_N(n)R^{[1,2]-}_{1,1}(m,n)|\rangle+\frac{\lambda}{N}\E\langle(x^{[1]}_N(m)x^{[2]}_N(n))^2\rangle^{[1,2]}\leq\frac{K(\lambda)}{\sqrt{N}},
	\end{split}
	\eeq
	Thus, comparing $\varphi(1)$ and $\psi_1(1)$, we get
	\[ \begin{split}
	\varphi(1) &=N\E\left[\langle x^{[1]}_N(m)x^{[2]}_N(n)\,R^{[1,2]-}_{1,1}(m,n)\rangle^{[1,2]}e^{\ii s\log\caL}\right]\\
	&=\lambda N\E\left[\langle (R^{[1,2]}_{1,1}(m,n))^2\rangle^{[1,2]}e^{\ii s\log\caL}\right]+\caO(N^{-1/2}).
	\end{split} 
	\]
	Similarly, for $\psi$, we also check that $\psi(0)=\E[X(0)]$, 
	\beq
	\sup_{0\leq t\leq 1}|\psi'(t)|=\caO(N^{-1/2}),
	\eeq 
	and
	\beq
	|\E[X(1)]-\E[X(0)]|\leq\sup_{0\le t\le 1}\E[|X'(t)|]\leq\frac{K}{\sqrt{N}}.
	\eeq
	We then find that 
	\begin{align}\label{eqn:self_second}
	\psi(1)&=\psi(0)+\caO(N^{-1/2})=\E[X(1)]+\caO(N^{-1/2})=\E [e^{\ii s\log\caL}]+\caO(N^{-1/2}).
	\end{align}
	Combining the results we obtained so far, we conclude that 
	\beq
	\begin{split}
		&\E\left[(N\langle R^{[1,2]}_{1,1}(m,n))^2\rangle^{[1,2]}-\langle x^{[1]}_N(m)^2x^{[2]}_N(n)^2\rangle^{[1,2]})e^{\ii s\log\caL}\right]
		\\&~~~=\lambda N\E\left[\langle R^{[1,2]}_{1,1}(m,n))^2\rangle^{[1,2]}e^{\ii s\log\caL}\right]+\caO(N^{-1/2})
	\end{split}\nonumber
	\eeq
	Furthermore, this relation and \eqref{eqn:self_second} implies that 
	\begin{equation*}
	\begin{split}
	\E\left[(N\langle (R^{[1,2]}_{1,1}(m,n))^2\rangle^{[1,2]}-\langle x^{[1]}_N(m)^2x^{[2]}_N(n)^2\rangle^{[1,2]})e^{\ii s\log\caL}\right]
	&=\frac{\lambda}{1-\lambda}\E [e^{\ii s\log\caL}]+\caO(N^{-1/2}).
	\end{split}
	\end{equation*}
	This completes the proof of Proposition \ref{prop:est}
\end{proof} 

Plugging the result of Proposition \ref{prop:est} into \eqref{eq:derivative_char}, we arrive at
\begin{align*}
&\phi'_N(\lambda)=\frac{\ii s-s^2}{4}\frac{(k_2-k_1)^2\lambda}{1-\lambda}\phi_N(\lambda)+\caO(N^{-1/2}),
\end{align*}
which proves Theorem \ref{thm:main} for the case when $k_1 \neq 0$ and $w_2=\infty$; see equation \eqref{eq:phi'}. 

As in the case of $k_1=0$ and $w_2<\infty,$ we can easily compute that the terms in the proof of Proposition \ref{prop:est} also follow the same self-consistent equations
\begin{align}
\E\left[\langle f\rangle^{[1,2]}e^{\ii s\log\caL}\right]
&=\frac{\lambda}{1-\lambda}\E [e^{\ii s\log\caL}]+\caO(N^{-1/2})
\end{align}
for any terms $f$ in \eqref{eq:derivative_char} and 
\begin{align}
\E\left[\langle g^2\rangle^{[1,2]}e^{\ii s\log\caL}\right]=\E [e^{\ii s\log\caL}]+\caO(N^{-1/2})
\end{align}
for any  $g=x^{[1]}_N(m)x^{[2]}_N(n),\;x^{[1],(1)}_N(m)x^{[1],(2)}_N(n),\;x^{[1]}_N(m)x^{\ast}_N(n)$ or $x^{[2],(1)}_N(m)x^{[2],(2)}_N(n).$

Further, it is easy to see that the derivative of the characteristic function contains the additional terms from the diagonal entries
\begin{equation*}\label{eq:derivative_char2}
\phi'_N(\lambda)=\phi_{N,o}(\lambda)+\phi_{N,d}(\lambda)
\end{equation*}
where
\begin{equation*}\begin{split}
\phi_{N,o}(\lambda)
&=\frac{s^2}{2}\sum_{m=1}^{k_1}\sum_{n=1}^{k_2}\E\left[(N\langle R^{[1,2]}_{1,1}(m,n)^2\rangle^{[1,2]}-\langle x^{[1]}_N(m)^2x^{[2]}_N(n)^2\rangle^{[1,2]})e^{\ii s\log\caL}\right]\\
&~~~-\frac{s^2}{4}\sum_{m=1}^{k_2}\sum_{n=1}^{k_2}\E\left[(N\langle R^{[2]}_{1,2}(m,n)^2\rangle^{[2]}-\langle x^{[2]}_N(m)^2x^{[2]}_N(n)^2\rangle^{[2]})e^{\ii s\log\caL}\right]\\
&~~~-\frac{s^2}{4}\sum_{m=1}^{k_1}\sum_{n=1}^{k_1}\E\left[(N\langle R^{[1]}_{1,2}(m,n)^2\rangle^{[1]}-\langle x^{[1]}_N(m)^2x^{[1]}_N(n)^2\rangle^{[1]})e^{\ii s\log\caL}\right]\\
&~~~-\frac{\ii s}{4}\sum_{m=1}^{k_2}\sum_{n=1}^{k_2}\E\left[(N\langle R^{[2]}_{1,2}(m,n)^2\rangle^{[2]}-\langle x^{[2]}_N(m)^2x^{[2]}_N(n)^2\rangle^{[2]})e^{\ii s\log\caL}\right]\\
&~~~+\frac{\ii s}{4}\sum_{m=1}^{k_1}\sum_{n=1}^{k_1}\E\left[(N\langle R^{[1]}_{1,2}(m,n)^2\rangle^{[1]}-\langle x^{[1]}_N(m)^2x^{[1]}_N(n)^2\rangle^{[1]})e^{\ii s\log\caL}\right]\\
&~~~+\frac{\ii s}{2}\sum_{m=1}^{k_2}\sum_{n=1}^{k_2}\E\left[(N\langle R^{[2]}_{1,*}(m,n)^2\rangle^{[2]}-\langle x^{[2]}_N(m)^2x^{*}_N(n)^2\rangle^{[2]})e^{\ii s\log\caL}\right]\\
&~~~-\frac{\ii s}{2}\sum_{m=1}^{k_1}\sum_{n=1}^{k_2}\E\left[(N\langle R^{[1]}_{1,*}(m,n)^2\rangle^{[1]}-\langle x^{[1]}_N(m)^2x^{*}_N(n)^2\rangle^{[1]})e^{\ii s\log\caL}\right]
\end{split}
\end{equation*}
and
\begin{equation*}
\begin{split}
&\phi_{N,d}(\lambda)
\\&=\frac{s^2}{w_2}\sum_{m=1}^{k_1}\sum_{n=1}^{k_2}\E\left[\langle x^{[1]}_N(m)^2x^{[2]}_N(n)^2\rangle^{[1,2]}e^{\ii s\log\caL}\right]\\
&\quad-\frac{s^2}{2w_2}\sum_{m=1}^{k_2}\sum_{n=1}^{k_2}\E\left[\langle x^{[2],(1)}_N(m)^2x^{[2],(2)}_N(n)^2\rangle^{[2]}e^{\ii s\log\caL}\right]\\
&\quad-\frac{s^2}{2w_2}\sum_{m=1}^{k_1}\sum_{n=1}^{k_1}\E\left[\langle x^{[1],(1)}_N(m)^2x^{[1],(2)}_N(n)^2\rangle^{[1]}e^{\ii s\log\caL}\right]\\
&\quad-\frac{\ii s}{2w_2}\sum_{m=1}^{k_2}\sum_{n=1}^{k_2}\E\left[\langle x^{[2],(1)}_N(m)^2x^{[2],(2)}_N(n)^2\rangle^{[2]}e^{\ii s\log\caL}\right]\\
&\quad+\frac{\ii s}{2w_2}\sum_{m=1}^{k_1}\sum_{n=1}^{k_1}\E\left[\langle x^{[1],(1)}_N(m)^2x^{[1],(2)}_N(n)^2\rangle^{[1]}e^{\ii s\log\caL}\right]\\
&\quad+\frac{\ii s}{w_2}\sum_{m=1}^{k_2}\sum_{n=1}^{k_2}\E\left[\langle x^{[2]}_N(m)^2x^{\ast}_N(n)^2\rangle^{[2]}e^{\ii s\log\caL}\right]-\frac{\ii s}{w_2}\sum_{m=1}^{k_1}\sum_{n=1}^{k_2}\E\left[\langle x^{[1]}_N(m)^2x^{\ast}_N(n)^2\rangle^{[1]}e^{\ii s\log\caL}\right].
\end{split}
\end{equation*}
Thus, we conclude that the characteristic function is asymptotically the solution of the following initial value problem with $\phi_N(0)=1$ 
\beq \label{eq:phi'_gen}
\phi_N'(\lambda) = \frac{\ii s-s^2}{4} \cdot \frac{(k_1-k_2)^2 \lambda}{1-\lambda}\phi_N(\lambda) + \frac{\ii s-s^2}{2w_2} \cdot (k_1-k_2)^2  \phi_N(\lambda) + O(N^{-\frac{1}{2}}).
\eeq
We omit the detail and complete the proof of Theorem \ref{thm:main}.


\section{Proof of Theorem \ref{thm:CLT} and \ref{thm:trans_CLT}}\label{app:CLT}
In Appendix \ref{app:CLT}, we prove the CLT for the LSS of spiked Wigner matrices. The proof of the CLT for the LSS is based on the strategy of \cite{Bai-Yao2005} in which the LSS is first written as a contour integral of the resolvent of a spiked Wigner matrix. Then, the averaged trace of the resolvent converges to a Gaussian process, which also implies that the limiting distribution of the LSS is Gaussian. 

It is the biggest obstacle in adapting the proof in \cite{Bai-Yao2005} for spiked Wigner matrices that the martingale CLT and covariance computation are hard to be reproduced with spikes; even with the special choice of rank-$1$ spike the proof for the CLT is very tedious as in \cite{Baik-Lee2017}. In \cite{Chung-Lee2019}, the interpolation between a general rank-$1$ spike and the special rank-$1$ spiked was introduced to compare the LSS, based on an ansatz that the mean and the variance of the LSS do not depend on the choice of the spike. In this paper, since we do not have a reference matrix to be compared with as in the rank-$1$ case, we introduce a direct interpolation between a spiked Wigner matrix of rank-$k$ and a Wigner matrix without any spikes. With the interpolation, we find the change of the mean in the limiting Gaussian distribution and also prove that its variance is invariant.

\begin{proof}[Proof of Theorem \ref{thm:CLT}]
	We adapt the proof of Theorem 5 in \cite{Chung-Lee2019} with the following change. Instead of interpolating the spiked Wigner matrices $M$ with the original signal and with the signal with all $1$'s considered in \cite{Baik-Lee2017}, we directly interpolate $M$ and $H$ and track the change of the mean. Consider the following interpolating matrix
	\[
	M(\theta) = \theta \sqrt{\lambda} X X^T + H
	\]
	and the corresponding eigenvalues $\{\mu_i(\theta)\}_{i=1}^{N}$ of $M(\theta)$ for $\theta \in [0, 1]$. 
	Let $\Gamma$ be a rectangular contour in the proof of Theorem 5 in \cite{Chung-Lee2019}. Applying Cauchy's integral formula, we have 
	\beq \label{eq:Cauchy_int}
	\sum_{i=1}^N f(\mu_i(1)) - N \int_{-2}^2 \frac{\sqrt{4-x^2}}{2\pi} f(x) \, \dd x = -\frac{N}{2\pi \ii} \oint_{\Gamma} f(z) \big( s_N(1,z) - s(z) \big) \dd z
	\eeq
	where $s(z) = \frac{-z + \sqrt{z^2 - 4}}{2}$ is the Stieltjes transform of the Wigner semicircle law and $s_N(\theta,z)$ is the Stieltjes transform of the empirical spectral distribution (ESD) of $M(\theta)$ for $\theta\in[0,1]$.
	Note that the normalized trace of the resolvent satisfies
	\beq
	\frac{1}{N} \Tr R(\theta,z) = \frac{1}{N} \sum_{i=1}^N \frac{1}{\mu_i(\theta) -z} = s_N(\theta,z)
	\eeq
	where $R(\theta,z)$ is the resolvent corresponding to $M(\theta)$, defined as
	\beq\label{resolvent}
	R(\theta, z) := (M(\theta) - zI)^{-1}
	\eeq
	for $z \in \C^+$ and $\theta\in[0,1]$. 
	
	The change of the mean in the CLT for $H$ and the CLT for $M$ can be computed by tracking the change of the corresponding resolvent in \eqref{resolvent}, since \eqref{eq:Cauchy_int} can be decomposed by
	\begin{align}
	\sum_{i=1}^N f(\mu_i(1)) - N \int_{-2}^2 \frac{\sqrt{4-x^2}}{2\pi} f(x) \, \dd x 
	&= -\frac{1}{2\pi \ii} \oint_{\Gamma} f(z) \big( \Tr R(1,z) - \Tr R(0,z) \big) \dd z\label{eq:Trace difference}\\
	&~~~~-\frac{1}{2\pi \ii} \oint_{\Gamma} f(z) \big( \Tr R(0,z) - Ns(z) \big) \dd z\label{null_case}
	\end{align} 
	and the fluctuation result of \eqref{null_case} is already given in \cite{Bai-Yao2005}.
	
	Set $\Gamma^{\varepsilon}=\{z\in\bbC:\min_{w\in\Gamma}|z-w|\leq\varepsilon\}.$ Choose $\varepsilon$ so that 
	\[
	\min_{w\in\Gamma^\varepsilon,x\in[-2,2]}|x-w|>2\varepsilon.
	\]
	Following the proof of Theorem 5 in \cite{Chung-Lee2019}, on $z\in\Gamma^{\varepsilon}_{1/2}:=\Gamma^{\varepsilon}\cap \{z\in\bbC:\,|\text{Im} z|>N^{-1/2}\}$, we first find that
	\begin{align}
	\frac{\partial}{\partial \theta} \Tr R(\theta, z) &= -\sum_{m=1}^k \sqrt{\lambda} \frac{\partial}{\partial z} \left( \bsx(m)^T R(\theta, z) \bsx (m) \right) = -k \frac{\partial}{\partial z} \left( \frac{\sqrt{\lambda} s(z)}{1+\theta \sqrt{\lambda} s(z)} \right) + O(N^{-\frac{1}{2}}) \nonumber\\
	&= -\frac{k \sqrt{\lambda} s'(z)}{(1+\theta \sqrt{\lambda} s(z))^2} + O(N^{-\frac{1}{2}})\label{eq:derivative_of_trace}
	\end{align}
	with high probability. More precisely, since the elementary resolvent expansion implies
	\beq
	\begin{split}
		R(0,z)-R(\theta,z)=\theta\sqrt{\lambda}R(\theta,z)\left(\sum_{\ell=1}^{k}\bsx(\ell)\bsx(\ell)^T\right)R(0,z),
	\end{split}\eeq
	we then find that
	\begin{align}\label{res_exp}
	\left( \bsx(m)^T R(0, z) \bsx (m) \right)=\left( \bsx(m)^T R(\theta, z) \bsx (m) \right)+\theta\sqrt{\lambda}\sum_{\ell=1}^k\left( \bsx(m)^T R(\theta, z) \bsx (\ell) \right)\left( \bsx(\ell)^T R(0, z) \bsx (m) \right)\nonumber.
	\end{align} 
	From the rigidity of the eigenvalues, we have a deterministic bound for resolvent
	\beq
	|\left( \bsx(m)^T R(\theta, z) \bsx (\ell) \right)|\leq \lVert R(\theta, z)\rVert\leq C.
	\eeq
	Since columns of spike $\{\bsx(\ell)\}_{\ell=1}^{k}$ are orthonormal, the isotropic local law for $R(0,z)$ implies that 
	\beq
	\left( \bsx(m)^T R(0, z) \bsx (\ell) \right)=s(z)\delta_{m\ell}+\caO(N^{-1/2})
	\eeq
	uniformly on $z\in\Gamma^\varepsilon$ (See Lemma A.1 of Supplementary Material for \cite{Chung-Lee2019}.)
	We then obtain that 
	\begin{align*}
	\left( \bsx(m)^T R(0, z) \bsx (m) \right)=\left( \bsx(m)^T R(\theta, z) \bsx (m) \right)\left[1+\theta\sqrt{\lambda}\left( \bsx(m)^T R(0, z) \bsx (m) \right)\right]+O(N^{-\frac{1}{2}})\nonumber
	\end{align*}
	and so
	\begin{align}
	\left( \bsx(m)^T R(\theta, z) \bsx (m) \right)=\frac{s(z)}{1+\theta\sqrt{\lambda}s(z)}+O(N^{-\frac{1}{2}}).\nonumber
	\end{align}
	This proves \eqref{eq:derivative_of_trace}.
	
	Moreover, on $\Gamma^{\varepsilon},$ we easily check that the exactly same argument holds for a finite rank perturbation of Wigner matrix (e.g. interlacing and rigidity properties). Thus, we conclude that \eqref{eq:Trace difference} is 
	\[\frac{k}{2\pi\ii}\int_\Gamma \frac{\sqrt{\lambda}s'(z)}{1+\sqrt{\lambda}s(z)}f(z)d z+o(1)\]
	with high probability.
	
	Finally, following the computation in the proof of Lemma 4.4 in \cite{Baik-Lee2017}, we then find that the difference between the LSS of $M$ and the LSS of $H$ is
	\beq
	k \sum_{\ell=1}^{\infty} \sqrt{\lambda^{\ell}} \tau_{\ell}(f).
	\eeq
	This proves the desired theorem.
\end{proof}
\begin{proof}[Proof of Theorem \ref{thm:trans_CLT}]
	We adapt the proof of Theorem 7 in \cite{Chung-Lee2019} with the following changes. Let $S$ be the variance matrix of the transformed matrix $\wt M.$ We then find that
	\[S_{ij}=\E[\wt M_{ij}^2]-(\E[\wt M_{ij}])^2=\frac1{N}+\lambda(\gh-\fh) (XX^T)_{ij}^2+\caO(N^{1-8\phi})\]
	and
	\[S_{ii}=\E[\wt M_{ii}^2]-(\E[\wt M_{ii}])^2=\frac{w_2}{N}+\lambda(\gh_d-\fh_d) (XX^T)^2_{ii}+\caO(N^{1-8\phi}).\]
	Normalizing and centering each entry of the matrix $\wt M$, we arrive at another Wigner matrix $W$ where
	\begin{align*}
	&W_{ij}=\frac{1}{\sqrt{NS_{ij}}}(\wt M_{ij}-\E\wt M_{ij}),&&W_{ii}=\sqrt{\frac{w_2}{NS_{ii}}}(\wt M_{ii}-\E\wt M_{ii}).\end{align*}
	Interpolating $W$ and $\wt M-\E[\wt M]$ by $W(\theta)=(1-\theta)W+\theta (\wt M-\E[\wt M])$, $W(\theta)$ is a general Wigner-type matrix with the corresponding quadratic vector equation
	\[-\frac{1}{m_{i}(\theta,z)}=z+\sum_{j=1}^{N}\E[W_{ij}(\theta)^2]\cdot m_j(\theta,z)\]
	where $m_i(\theta,z)\delta_{ij}$ is the limiting distribution of the $(i,j)$-element of the resolvent 
	\[R^W(\theta,z)=(W(\theta)-zI)^{-1}\]
	for $0\leq\theta\leq1.$
	Recall the $s(z)$ is the Stieltjes transform of the Wigner semicircle law. We also directly check that $m_{i}(\theta,z)=s(z)+\caO(N^{-2\phi}).$
	Moreover, the anisotropic local law for the general Wigner-type matrix implies that uniformly on $z\in\Gamma^{\varepsilon}_{1/2}$
	\[(\bsx(m)^TR^{W}(\theta,z)\bsx(\ell))=s(z)\delta_{m\ell}+\caO(N^{-1/2})\]
	(See Lemma D.1 of Supplementary Material for \cite{Chung-Lee2019}.)
	
	Following the proof of Lemmas B.2 and B.3 in \cite{Chung-Lee2019}, we check that 
	\begin{itemize}
		\item Uniformly on $z\in\Gamma^{\varepsilon}_{1/2},$
		\[\Tr R^W(1,z)-\Tr R^W(0,z)=k\lambda(\gh-\fh)s'(z)s(z)+\caO(N^{3/2}N^{-4\phi})\]
		\item Uniformly on $z\in\Gamma^{\varepsilon}\backslash\Gamma^{\varepsilon}_{1/2},$
		\[|\Tr R^W(1,z)-\Tr R^W(0,z)|=\caO(N^{1/3}).\]
	\end{itemize} 
	Our next step is to consider $\wt M=W(1)+\E[\wt M].$ Since 
	\[
	\wt M=W(1)+\sqrt{\lambda \fh}XX^T+\diag(d_1,\cdots,d_N)
	\]
	where $d_i=\E[\wt M_{ii}]-\sqrt{\lambda\fh}(XX^T)_{ii},$ we then find that
	\[ \begin{split}
	&\Tr(\wt M-zI)^{-1}-\Tr R^W(0,z)\\
	&=k\lambda(\gh-\fh)s'(z)s(z)-\frac{k\sqrt{\lambda\fh}s'(z)}{1+\sqrt{\lambda\fh}s(z)}-k\sqrt{\lambda}(\sqrt{\fh_d}-\sqrt{\fh})s'(z)+O(N^{-1/2})
	\end{split} \]
	uniformly on $z\in\Gamma^{\varepsilon}_{1/2}.$
	Thus, we obtain the desired CLT by applying Cauchy's integral formula as in the proof of Theorem \ref{thm:CLT}.
\end{proof}
\end{document}